%% file: shimmura.tex
\documentclass[a4]{article} 
\usepackage{amsmath,amssymb,amsthm}
\usepackage{bm}
\usepackage{tikz}

\usepackage{wrapfig}
\usepackage{ascmac}
\usepackage{makeidx}
\usepackage{comment}
\usepackage[margin=3cm]{geometry}
\usepackage{usebib}
\usepackage{algorithm}
\usepackage{algorithmic}
\usepackage[pdftex]{hyperref}

\usepackage{natbib}
\usepackage{authblk}

\makeindex

\newtheorem{teiri}{Theorem}

\newtheorem{hodai}{Lemma}

\newcommand{\argmax}{\mathop{\rm argmax}\limits}
\newcommand{\argmin}{\mathop{\rm argmin}\limits}
\newcommand{\prox}{{\rm prox}}

\newcommand{\Rb}{{\mathbb R}}

\newcommand{\trace}{{\rm trace}}
\newcommand{\other}{{\rm otherwise}}
\newcommand{\for}{{\rm for}}
\newcommand{\ifq}{{\rm if}\quad}
\newcommand{\mate}{{\mathcal E}}

\newcommand{\Cm}{{\rm C}}

\newfont{\bg}{cmr9 scaled\magstep4}
\newcommand{\bigzerol}{\smash{\lower1.0ex\hbox{\bg 0}}}
\newcommand{\bigzerou}{%
	\smash{\hbox{\bg 0}}}

\providecommand{\keywords}[1]
{
	\small	
	\textbf{Keywords---} #1
}

	\title{Converting ADMM to a Proximal Gradient \\for Efficient Sparse Estimation}		
	\author{Ryosuke Shinmura and Joe Suzuki}												
	\date{\today}

\begin{document}

\maketitle

%
%
%
%

\abstract{
	In sparse estimation, such as fused lasso and convex clustering, we apply either the proximal gradient method or the alternating direction method of multipliers (ADMM) to solve the problem. It takes time to include matrix division in the former case, while an efficient method such as FISTA (fast iterative shrinkage-thresholding algorithm) has been developed in the latter case.
	This paper proposes a general method for converting the ADMM solution to the proximal gradient method, assuming that assumption that the derivative of the objective function is Lipschitz continuous. Then, we apply it to sparse estimation problems, such as sparse convex clustering and trend filtering, and we show by numerical experiments that we can obtain a significant improvement in terms of efficiency.}


\maketitle
\section{Introduction}
\label{intro}
We consider sparse estimation, in particular for the least absolute shrinkage and selection operator (lasso  \citealt{lasso}).
Suppose that we have gene expression data for breast cancer patients. Out of a set of $p=10000$ genes, we wish to identify the genes that determine each patient's positive or negative status from the case and control data ($n=100$ in total).
The conventional statistical approaches are not applicable to cases in which $p$ is enormous compared with the sample size $n$, such as the considered case.
Sparse estimation seeks to identify significant covariates that relate to the predicter.
For linear regression of $p$ variables (we assume the intercept to be zero), from $n$ examples,
we find the values of the coefficient $\beta$ that minimize the square error $\|y-X\beta\|^2$ for $X\in {\mathbb R}^{n\times p}$ and $y\in {\mathbb R}^n$.
However, if $p$ is too large, we are often tempted to neglect variables that are not important. In sparse estimation, assuming that each column of $X$ is normalized and choosing a positive constant $\lambda$ if the absolute value of $\beta_j$ that is obtained by least squares regression is smaller than $\lambda$, we may set the value to zero. Otherwise, the absolute value is reduced by $\lambda$. The formal definition is to find $\beta$ that minimizes
$$\|y-X\beta\|^2_2+\lambda \|\beta\|_1\ ,$$
where $\|\beta\|_1:=\sum_{i=1}^p \lvert \beta_j \rvert$.
Since the two terms are convex, the whole problem is convex. The definition of convexity is given in Section \ref{convex}. One of the main reasons that lasso is so famous is that there is an efficient procedure for finding the solution because the optimization is convex.

However, suppose we use an information criterion such as Akaike's information criterion (AIC)\citep{akaike1974new} or the Bayesian information criterion (BIC)\citep{schwarz1978estimating}.
In that case, the second term is $\sum_{j=1}^p I[\beta_i\not=0]$ (the number of variables) times a constant 
such as $2$ or $\log n$
, where $I[A]$ is one and zero when the condition $A$ is true and false, respectively. The function $f(x)=I[x\not=0]$ violates convexity \citep{Suzuki21}.
Thus, the second term in the definition of the information criterion is nonconvex,
and we require all $2^p$ combinations to find the model that minimizes the AIC or BIC.

Lasso includes variants\citep{hastie2019statistical} such as logistic regression, Poisson regression, Cox regression and linear regression and extensions such as group lasso\citep{Yuan2006}, fused lasso\citep{fused}, graphical lasso, and convex clustering\citep{Hocking2011, Lindsten2011, Pelckmans2005}.
Each formulates the problem as a convex optimization problem, and obtaining solutions is efficient.
They realize model selection by setting to zero the parameter values that are not significant.

Thus, sparse estimation is currently being used in many fields such as image recognition and survival time analysis. As more and more data becomes available and the dimension of the data increases, finding variables that are strongly related to sparse estimation and computational efficiency due to the large number of $0$ will become more important. However, while efficient methods have been proposed for each problem such as lasso as a release, ADMM used for general problems is not efficient for some problems, so efficient algorithms are needed.

In this paper, we focus on the efficiency of sparse estimation procedures.
In particular, we are motivated by the following observation.
Two main approaches to finding the solution of lasso are the proximal gradient method and ADMM (alternating direction method of multipliers)\cite{BoydADMM,Gabay1976}.
For detailed descriptions of the two procedures, see Section 3.
We wonder why some procedures, such as fused lasso and graphical lasso, use ADMM, while others, such as group lasso and convex clustering, use the proximal gradient method.
It seems that the proximal gradient method is more efficient than ADMM because efficient modifications such as the fast iterative shrinkage-thresholding algorithm (FISTA)\citep{Beck2009} can be easily used for the former whereas inverse matrix computation is inevitable for the latter. The main contribution of this paper is the following claim:

\fbox{
	\begin{minipage}{110mm}
		The sparse estimation procedure that is realized by ADMM can be transformed to a sparse estimation procedure that is realized by the proximal gradient method as long as its Lipschitz constant exists.
	\end{minipage}
}

This implies that sparse estimation will be improved if the proximal gradient-based procedure with a Lipschitz constant is more efficient than the ADMM-based procedure.

The Lipschitz condition is satisfied by adding a regularization term such as the $L_1,L_2$-norm to the loss function of linear regression, such as Lasso, Sparse group lasso, Sparse convex clustering, and trend filtering. The Lipschitz condition is also satisfied by adding a regularization term to the loss function of logistic loss or cox regression, which is used when the target variable is binary or multilevel.

The remainder of this paper is organized as follows. Section 2 presents work that is related to the results in this paper. Section 3 presents background knowledge for understanding this paper. Section 4 derives a general method for converting a problem that is solved by ADMM to a problem that is solved by the proximal gradient method. Sections 5 and 6 apply it to sparse convex clustering\citep{Wang2018} and trend filtering\citep{Kim2009} to evaluate its performance. Finally, Section 7 summarizes the results of this paper and describes future work.

\section{Related Work}

The optimization problem that is considered in this paper is
\begin{equation}
	\label{hondai}
	\min_x \; f(x)+g(x)+h(Ax)\ .
\end{equation}
for convex $f,g:\Rb^n\rightarrow\Rb,h:\Rb^m\rightarrow\Rb$ and $A\in\Rb^{m\times n}$,
where $f$ is differentiable, $\nabla f$ is Lipschitz continuous with parameter $L_f>0$, and $h$ is a closed convex function.
For example, in sparse convex clustering\citep{Wang2018}, $f$ is the loss function, and $g,h$ are regularization terms (constraints).
The dual problem for (\ref{hondai}) is
\begin{align}
	\label{dual}
	\min_{y\in \Rb^m}\; (f+g)^*(-A^Ty) + h^*(y)\ ,
\end{align}
where $(f+g)^*,h^*$ are the conjugate functions of $f+g,h$.

For optimization problems such as (\ref{hondai}), sparse estimation often uses the proximal gradient method when, for example,
$h\equiv 0$, and ADMM otherwise.
Although the implementation of ADMM is simple and it can be applied to various problems,
it is often computationally expensive.
For example, we often need to compute the inverse matrix to solve the optimization problem.
To simplify the computation, there are generalized ADMM \citep{APGM,deng2016global} which apply the proximal gradient method to ADMM. Nevertheless, the convergence becomes slow when $n$ is large.

In addition, we may use the alternating minimization algorithm (AMA)\citep{ama,Davis-yinsp}, which slightly modifies the ADMM steps.
We can regard it as an application of the proximal gradient method to the dual problem (\ref{dual}).
However, this requires either $f$ or $g$ to be strongly convex and have narrow applicability.
In addition, when row $m$ of $A$ is large, it becomes a proximal gradient method with many dimensions, and convergence becomes slow.

In this paper, we apply the proximal gradient method to method of multiplier \citep{Rockafellar1976}
to convert the ADMM problem to a proximal gradient method problem and solve it. Since the proposed method applies the proximal gradient method to the main problem,
it can solve the problem quickly when, for example, $m$ is large.

\section{Preliminaries}
\label{preliminary}

\subsection{Convex function and its subdifferential}
\label{convex}
Function $f:\Rb^n \rightarrow \Rb$ is {\it convex} if
\begin{align}
	\label{totsusei}
	f((1-\lambda)x+\lambda y)\leq(1-\lambda)f(x)+\lambda f(y)
\end{align}
for any $x,y\in \Rb^n$ and $0\leq \lambda \leq 1$.
In particular, if no equality holds in (\ref{totsusei}) for any $x,y\in \Rb^n$ and $0<\lambda<1$,
the function $f$ is {\it strongly convex}.
Moreover, a convex function $f:\Rb^n \rightarrow \Rb$ is {\it closed} if
$\{x\in {\mathbb R}^n\vert f(x)\leq \alpha \}$ is a closed set for each $\alpha\in {\mathbb R}$.

For convex function $f:\Rb^n\rightarrow \Rb$, the set of $z\in \Rb^n$ such that
\begin{align}
	\label{retsubi}
	f(x)\geq f(x_0)+\langle z, , x-x_0 \rangle
\end{align}
for any $x\in \Rb^n$ is the {\it subdifferential} of $f$ at $x_0\in \Rb^n$ and written as $\partial f(x_0)$.
For example, the subdifferential of $f(x)=\lvert x \rvert,x\in \Rb$ at $x=0$ is the set of $z$ such that $\lvert x \rvert \geq \lvert zx \rvert ,x\in \Rb$,
and we write $\partial f(0)=\{z\in \Rb\mid \lvert z \rvert \leq 1\}$.

\subsection{ADMM}
Let $A\in\Rb^{p\times n},B\in\Rb^{p\times m},c\in \Rb^p$, and $f,g:\Rb^n \rightarrow \Rb,h:\Rb^m \rightarrow \Rb$ be convex.
We consider the convex optimization
\begin{align}
	\label{admm}
	\min_{x,y} \;\;f(x)+g(x)+h(y)\\
	{\rm subject \,\,\, to} \,\,\, Ax+By = c\ .\nonumber
\end{align}
When we apply the ADMM, for the convex optimization formulated as in (\ref{admm}),
we define the augmented Lagrangian
\begin{align}
	\label{admmlag}
	L_\nu(x,y,\lambda)=f(x)+g(x)+h(y)+\langle \lambda,Ax+By-c \rangle +\frac{\nu}{2}\| Ax+By-c\|^2
\end{align}
for $\nu>0$, and repeatedly update via the equations
\begin{align*}
	x^{(k+1)}&=\argmin_{x}\; L_\nu(x,y^{(k)},\lambda^{(k)})\\
	y^{(k+1)}&=\argmin_{y}\; L_\nu(x^{(k+1)},y,\lambda^{(k)})\\
	\lambda^{(k+1)}&=\lambda^{(k)}+\nu(Ax^{(k+1)}+By^{(k+1)}-c) 
\end{align*}
from the initial values $y^{(0)}$ and $\lambda^{(0)}$ until convergence to obtain the solution.

\subsection{Proximal Gradient Method}
The proximal gradient method finds the minimum solution of $F$ expressed by the sum of convex functions $f,g$ such that
$f$ is differentiable and $g$, which is not necessarily differentiable.
We define the functions
\begin{align}
	\label{kinsetuq}
	Q_\eta(x,y)&:=f(y)+\langle x-y,\nabla f(y)\rangle +\frac{1}{2\eta}\|x-y\|^2+g(x)\\
	\label{kinsetup}
	p_\eta(y)&:=\argmin_x \;\; Q_\eta (x,y)
\end{align}
for $\eta>0$, and generate the sequence $\{x_k\}$ via
\begin{equation}
	x_{k+1} \leftarrow p_\eta(x_k)
\end{equation}
from the initial value $x_0$ until convergence to obtain the solution.
If we define the {\it proximal map} w.r.t. $g:\Rb^n\rightarrow \Rb$ by
\begin{align}
	\label{kinsetusyazou}
	\prox_{g}(y)=\argmin_x \left\{g(x)+\frac{1}{2}\|y-x\|^2_2\ \right\}\, 
\end{align}
then (\ref{kinsetup}) can be expressed by
\begin{align}
	p_{\eta}(y)&=\argmin_x Q_\eta (x,y)\nonumber \\
	&=\argmin_x \left\{\langle x-y, \nabla f(y)\rangle +\frac{1}{2\eta}\|x-y\|^2+g(x)\right\} \nonumber \\
	&=\argmin_x \left\{g(x)+\frac{1}{2\eta}\|x-y-\eta \nabla f(y)\|^2\right\} \nonumber \\
	\label{kinprox}
	&=\prox_{\eta h}(y-\eta \nabla f(y)) \ 
\end{align}
In each iteration, the proximal gradient seeks $x$ that minimizes
the sum of the quadratic approximation of $g(x)$ around $x_k$ and $h(x)$.
The ISTA (iterative shrinkage-thresholding algorithm) procedure obtains $O(k^{-1})$ accuracy for the number of updates $k$ \citep{Beck2009}.
We may replace ISTA by the faster procedure below: using the sequence $\{\alpha_k\}$ such that
$\alpha_0=1,\alpha_{k+1}=\frac{1+\sqrt{1+4\alpha_t^2}}{2}$, generates $\{x_k\}$ and $\{y_k\}$ via the equations
\begin{align*}
	x_{k}&=p_{\eta}(y_k) \\
	y_{k+1}&=x_k+\frac{\alpha_k-1}{\alpha_{k+1}}(x_k-x_{k-1})
\end{align*}
from the initial value $y_1=x_0$ until convergence to obtain the solution.
Note that the quantity $\frac{\alpha_k-1}{\alpha_{k+1}}$ is zero when $k=1$, increases with $k$, and converges to one  as $k\rightarrow \infty$.
It behaves similarly to ISTA  when $k$ is small, and accelerates the updates when $k$ increases to gain efficiency.
The FISTA (fast iterative shrinkage-thresholding algorithm) procedure obtains $O(k^{-2})$ accuracy for the number of updates $k$ \citep{Beck2009}.
This paper mainly uses the FISTA.

Even if it is updated using the formulas given by ISTA and FISTA,
they do not necessarily converge to $x$, which minimizes the objective function $f(x)$ unless
we choose an appropriate parameter $\eta$.
In the following, we assume that $\nabla f$ is Lipschitz continuous, which means that there exists $L_f>0$ such that for arbitrary $x,y$,
\begin{align}
	\label{lipschitz}
	\| \nabla f(x)-\nabla f(y) \|\leq L_f \| x -y \|\ .
\end{align}
It is known that ISTA and FISTA converge to $x$ that minimizes $F(x)$
if we choose $\eta>0$ as $0<\eta\leq \frac{1}{L}$ \citep{Beck2009}.

\section{The Proposed Method}
\label{yaru}
(\ref{hondai}) is equivalent to the following:
\begin{align}
	\label{seiyaku}
	\min_{x,y} \; f(x)+g(x)+h(y)\\
	{\rm subject \,\,\, to} \,\,\, Ax = y\nonumber\ .
\end{align}
If we apply ADMM, then the augmented Lagrangian (\ref{admmlag}) for (\ref{seiyaku}) is
\begin{align}
	\label{honkaku}
	L_\nu(x,y,\lambda)=f(x)+g(x)+h(y)+\langle \lambda,Ax-y \rangle +\frac{\nu}{2}\| Ax-y\|^2
\end{align}
for $\nu>0$.

In the proposed method, we update $x,y$ simultaneously via the
\begin{align}
	\label{teian}
	(x^{(k+1)},y^{(k+1)})&= \argmin_{x,y}\{L_{\nu^{(k)}}(x,y,\lambda^{(k)})\} \\
	\lambda^{(k+1)}&=\lambda^{(k)}+(Ax^{(k+1)}-y^{(k+1)}) \nonumber 
\end{align}
from the initial value $\lambda^{(0)}$. Although, in general, changing the value $\nu$ for each $k$ may improve the performance, we set $\nu$ to be constant to proceed with the derivation, making the notation simple.

To update via (\ref{teian}), we consider the minimization of
\begin{align*}
	\phi(x):&=\min_y L_{\nu}(x,y,\lambda^{(k)})\\
	&=f(x)+ g(x)+\min_y \{h(y)+\langle \lambda^{(k)}, Ax-y\rangle +\frac{\nu}{2}\|Ax-y\|^2\}
\end{align*}
w.r.t. $y$.

\begin{teiri}
	If we define $\phi_1(x):=f(x)+\min_y \{h(y)+\langle \lambda^{(k)}, Ax-y\rangle +\frac{\nu}{2}\|Ax-y\|^2\}$, then $\phi_1$ is differentiable and we have 
	\begin{align}
		\label{nabla1}
		\nabla\phi_1(x) = \nabla f(x)+A^T(\prox_{\nu h^*}(\nu Ax +\lambda^{(k)})).
	\end{align}
\end{teiri}

\begin{proof}
	we define the function $\psi(x)$ obtained by removing $f(x),g(x)$ from $\phi(x)$:
	\begin{align}
		\psi(x):&=\min_y \{h(y)+\langle \lambda^{(k)} , Ax-y\rangle +\frac{\nu}{2}\|Ax-y\|^2\} \nonumber\\
		&=\min_y\{h(y)+\frac{\nu}{2}\|y\|^2-\langle y,\nu Ax+\lambda^{(k)}\rangle\}+ \langle \lambda^{(k)},Ax\rangle +\frac{\nu}{2}\|Ax\|^2\label{145}\\
		&=-\max_y\{\langle y,\nu Ax+\lambda^{(k)} \rangle -h(y)-\frac{\nu}{2} \|y\|^2\}+\langle \lambda^{(k)},Ax\rangle +\frac{\nu}{2}\|Ax\|^2\nonumber\\
		\label{psi}
		&=-r^*(\nu A x+\lambda^{(k)})+\langle \lambda^{(k)},Ax\rangle +\frac{\nu}{2}\|Ax\|^2\ ,
	\end{align}
	where $r(u):=h(u)+\frac{\nu}{2}\|u\|^2$ and $r^*(v):=\sup_{u}\{\langle u,v \rangle- r(u)\}$.
	Because the first term of (\ref{145}) can be written as 
	$$
	\min_y\{h(y)+\frac{\nu}{2}\|y-(Ax+\frac{\lambda^{(k)}}{\nu})\|^2-\frac{\nu}{2}\|Ax+\frac{\lambda^{(k)}}{\nu}\|^2\}\ ,
	$$
	the quantity $h(y)+\langle \lambda^{(k)}, Ax-y\rangle +\frac{\nu}{2}\|Ax-y\|^2$ is minimized when
	\begin{align}
		\label{ycousin}
		y^*(x) = \prox_{\nu^{-1}h}(Ax+\nu^{-1}\lambda^{(k)})\, 
	\end{align}
	where $\prox_{\nu^{-1}h}(\cdot)$ is the proximal map defined in (\ref{kinsetusyazou}).
	Then, we notice the following lemma:
	\begin{hodai}[\citealt{rockafellar1970} Theorem 26.3]
		\label{conj}
		Assume that $s:\Rb^m\rightarrow \Rb$ is closed and strongly convex.
		Then, conjugate function $s^*$ is differentiable and $\nabla s^*(v)=\argmax_{u\in \Rb^m}\{\langle u,v \rangle -s(u)\}$ for $v\in \Rb^m$.
	\end{hodai}
	
	From Lemma\ref{conj}, we have
	\begin{align}
		\nabla r^*(v) &= \argmax_u\{\langle u,v \rangle-r(u)\}=\argmax_u\{\langle u,v \rangle-h(u)-\frac{\nu}{2}\|u\|^2\}\nonumber\\
		&=\argmin_u\{\frac{1}{2}\|u\|^2+\frac{1}{\nu}h(u)-\langle\frac{v}{\nu},u\rangle\} \nonumber\\
		&=\argmin_u\{\frac{1}{2}\|u-\frac{v}{\nu}\|^2 + \frac{1}{\nu}h(u)\}\nonumber\\
		\label{conjugate}
		&=\prox_{h/\nu}(\frac{v}{\nu})\ .
	\end{align}
	If we substitute $v=\nu A x+\lambda^{(k)}$ into (\ref{conjugate}), we have
	\begin{align}
		\label{conjnab}
		\nabla r^*(\nu Ax+\lambda^{(k)})=\nu A^T \prox_{h/\nu}(Ax +\nu^{-1}\lambda^{(k)})\ .
	\end{align}
	Moreover, we notice another lemma:
	\begin{hodai}[\citealt{Moreau1965}]
		\label{moreau}
		If the function $s:\Rb^m\rightarrow \Rb$ is convex, then for any $z\in\Rb^m$ and $\gamma>0$, we have
		\[
		\prox_{\gamma s}(z) +\gamma\prox_{s^*/\gamma}(\gamma^{-1}z)=z 
		\]\ .
	\end{hodai}
	
	From Lemma \ref{moreau}, (\ref{psi}), and (\ref{conjnab}), we have
	\begin{align*}
		\nabla\phi_1(x)&=\nabla f(x) +A^T\lambda^{(k)} +\nu A^TAx -\nu
		A^T\prox_{h/\nu}(Ax+\nu^{-1}\lambda^{(k)})\\
		&=\nabla f(x)+A^T(\prox_{\nu h^*}(\nu Ax +\lambda^{(k)}))\ .
	\end{align*}
\end{proof}
Since $\phi(x)=f(x)+g(x)+\psi(x)=\phi_1(x)+g(x)$ can be expressed by the sum of differentiable $\phi_1(x)$ and nondifferentiable $g(x)$,
the minimization can be solved via the proximal gradient: update each time
via\qed
\begin{align}
	\label{proximal}
	x^{(l+1)} = \prox_{\eta g}(x^{(l)}-\eta \nabla \phi_1(x^{(l)}))
\end{align}
(see (\ref{kinprox})),
where parameter $\eta>0$ is $\eta\leq \frac{1}{L}$ for $L>0$ such that
\[
\|\nabla \phi_1(x_1)-\nabla \phi_1(x_2)\|\leq L\|x_1-x_2\|
\]

Then, convergence is guaranteed.

\begin{hodai}
	\label{nonex}
	If the function $h:\Rb^m\rightarrow \Rb$ is convex, then
	\[
	\|\prox_h(x)-\prox_h(y)\|\leq \|x-y\| \quad \for \quad \forall x,y\in\Rb^m\ .
	\]
\end{hodai}
\begin{proof}
	Let $u:=\prox_h(x),v:=\prox_h(y)$.
	Then, because $u$ minimizes $h(u)+\frac{1}{2}\|x-u\|^2$, if we subdifferentiate it by $u$ and equate it to be zero,
	there exists $s\in \partial h(u)$ such that $s+u-x=0$. Similarly, we have $t:=y-v$ is in $\partial h(v)$. 
	Because  the convexity of $h$ means  $\langle u-v,s -t \rangle \geq 0$, we have 
	\begin{align*}
		\|x-y\|^2 &= \|u+s - (v+t)\|^2 \\
		&=\|u-v\|^2 + 2 \langle u-v , s-t \rangle + \|s-t\|^2\geq \|u-v\|^2\ .
	\end{align*}
\end{proof}

\begin{teiri}
	\label{teirilip}
	If $\;\nabla f$ is Lipschitz continuous with parameter $L_f$, then for any $x_1,x_2 \in \Rb^n$, we have 
	\begin{align}
		\label{teiriL}
		\|\nabla \phi_1(x_1)-\nabla \phi_1(x_2)\|\leq (L_f+\nu\lambda_{\max}(A^TA))\|x_1-x_2\|\ .
	\end{align}
\end{teiri}
\begin{proof}
	From Lemma \ref{nonex}, for $x_1,x_2 \in \Rb^n$, we have
	\begin{align}
		&\quad \|A^T(\prox_{\nu h^*}(\nu Ax_1 +\lambda^{(k)}))-A^T(\prox_{\nu h^*}(\nu Ax_2 +\lambda^{(k)}))\| \nonumber \\
		&\leq \sqrt{\lambda_{\max}(A^TA)} \times \|\prox_{\nu h^*}(\nu Ax_1 +\lambda^{(k)})-\prox_{\nu h^*}(\nu Ax_2 +\lambda^{(k)})\| \nonumber \\
		&\leq \sqrt{\lambda_{\max}(A^TA)} \times \|\nu Ax_1 +\lambda^{(k)}-(\nu Ax_2 +\lambda^{(k)})\| \nonumber \\
		&\leq \nu \lambda_{\max}(A^TA) \|x_1-x_2\|\ .
	\end{align}
	Thus, when $\nabla f$ is Lipschitz continuous with parameter $L_f$, we have
	\begin{align}
		\label{kinsetupara}
		\|\nabla \phi_1(x_1)-\nabla \phi_1(x_2)\|\leq (L_f+\nu\lambda_{\max}(A^TA))\|x_1-x_2\|\ ,
	\end{align}
\end{proof}
In Theorem \ref{teirilip}, the proximal gradient converges for $\eta := 1/(L_f+\nu \lambda_{\max}(A^TA))$.
Hence, it is possible to solve (\ref{teian}) efficiently
when $\nabla f$ is Lipschitz continuous.

The procedure (\ref{proximal}) is not as efficient as FISTA \citep{Beck2009}.
We show the modification to FISTA in Algorithm \ref{alg:fista}.

\begin{algorithm}[h]
	\caption{(FISTA for $\min \phi(x)$)  \newline Input :$z^{(0)}$, output :$z^{(\infty)}$}
	\label{alg:fista}
	\begin{enumerate}
		\item Initialize $u^{(1)}=z^{(0)},\eta\in (0,1),\alpha_1=1$. 
		
		For $j=1,2,\ldots$
		\item (Update $z$)
		\[
		z^{(j)} = \prox_{\eta g}(u^{(j)}-\eta \nabla \phi_1(u^{(j)}))
		\]
		\item (Update $\alpha$ and $u$)
		\begin{align*}
			\alpha_{j+1} &=\frac{1+\sqrt{1+4\alpha_j^2}}{2}\\
			u^{(j+1)}&=z_j+\frac{\alpha_j-1}{\alpha_{j+1}}(z_j-z_{j-1})
		\end{align*}
		\item Repeat Steps 2-3 until convergence to obtain $z=z^{(\infty)}$.
	\end{enumerate}
\end{algorithm}

Similarly, if we put $z=\nu^{-1}\lambda^{(k)}+Ax,\gamma=\nu^{-1}$ in Lemma \ref{moreau}, from (\ref{ycousin}), we obtain
\begin{align}
	\lambda^{(k)}+\nu(Ax-y^*(x)) = \prox_{\nu h^*}(\nu Ax +\lambda^{(k)})\ .
\end{align}
Thus, the update of $\lambda$ is
\begin{align}
	\lambda^{(k+1)}=\prox_{\nu h^*}(\lambda^{(k)} + \nu Ax^{(k+1)})
\end{align}
and we do not have to update the value of $y$ because it is not required to update $x,\lambda$.

We show the actual procedure in Algorithm \ref{alg:S-CD}.

\begin{algorithm}[h]
	\caption{(Proposed Method for solveing (\ref{seiyaku})) \newline Input: $x^{(1)},\lambda^{(1)}$, output: $x^{(\infty)},\lambda^{(\infty)}$}
	\label{alg:S-CD}
	\begin{enumerate}
		\item Initialize $\nu>0$. 
		
		For $k=1,2,\ldots$
		\item (Update $x$)
		
		Give  $x^{(k)}$ as input to Algorithm \ref{alg:fista} and take as output the value $x^{(k+1)}$.
		\item (Update $\lambda$)
		\[
		\lambda^{(k+1)} = \prox_{\nu h^*}(\lambda^{(k)} + \nu Ax^{(k+1)})
		\]
		\item Repeat Steps 2-3 until convergence to obtain $x=x^{(\infty)}$ and $\lambda=\lambda^{(\infty)}$.
	\end{enumerate}
\end{algorithm}

\section{Application to Sparse Convex Clustering}
\label{sp}
Let $X_{1\cdot},X_{2\cdot},\ldots,X_{n\cdot}\in \Rb^p$ be the $n$ observations w.r.t. $p$ variables.
Let $U_{i\cdot}$ and $u_j$ be the row and column vectors of a matrix $U\in \Rb^{n\times p}$.

The optimization of sparse convex clustering \citep{Wang2018} is formulated as follows:
\begin{align}
	\label{scvxsub}
	\min_{U}\;\frac{1}{2}\sum_{i=1}^n\|X_{i\cdot}-U_{i\cdot}\|_2^2+\gamma_1\sum_{(i,j)\in {\mathcal E}}w_{(i,j)}\|U_{i\cdot}-U_{j\cdot}\|_2 
	+\gamma_2 \sum_{j=1}^p r_j \|u_j\|_2 \ ,
\end{align}
where $\gamma_1,\gamma_2\geq 0$ are the regularized parameters, $w_{(i,j)}$ and $r_j\geq 0$ are nonnegative constants (weights),
and 
${\mathcal E} = \{(i,j);w_{ij}>0,1 \leq i < j \leq n \} $.

The objective function is the sum of the convex clustering's objective function and the group lasso regularization term.
Since all the elements associated with $u_j$ are expected to become zeros simultaneously when $\gamma_2$ is large,
sparse convex clustering can choose relevant variables for clustering.

To apply the proposed method, we rewrite (\ref{scvxsub}) as follows.
\begin{align}
	\label{scvxseiyaku}
	&\min_{U}\;\frac{1}{2}\sum_{i=1}^n\|X_{i\cdot}-U_{i\cdot}\|_2^2+\gamma_1\sum_{(i,j)\in {\mathcal E}}w_{(i,j)}\|v_{(i,j)}\|_2 
	+\gamma_2 \sum_{j=1}^p r_j \|u_j\|_2   \\
	& {\rm subject \quad to} \quad U_{i \cdot}-U_{j \cdot} - v_{(i,j)}=0 \quad ((i,j)\in {\mathcal E})\nonumber
\end{align}
We note that the optimization with the constraints above is equivalent to the minimization of the augmented Lagrangian below:
\begin{align}
	L_{\nu} (U,V,\Lambda)=&\frac{1}{2}\sum_{i=1}^n\|X_{i\cdot}-U_{i\cdot}\|_2^2+\gamma_1\sum_{(i,j)\in {\mathcal E}}w_{(i,j)}\|v_{(i,j)}\|_2 
	+\gamma_2 \sum_{j=1}^p r_j \|u_j\|_2   \nonumber \\
	&+\sum_{(i,j)\in {\mathcal E}}\langle\lambda_{(i,j)},v_{(i,j)}-U_{i \cdot}+U_{j \cdot}\rangle +\frac{\nu}{2}\sum_{(i,j)\in {\mathcal E}}\|v_{(i,j)}-U_{i \cdot}+U_{j \cdot}\|_2^2  \nonumber 
\end{align}

\subsection{Application of the Proposed Method}
In the following, we define $A_{\mathcal E}$ by
$A_{\mathcal E}U=(u_{i,k}-u_{j,k})_{(i,j)\in {\mathcal E},k=1,\ldots,p}$, and denote $\langle B,C \rangle =\trace(B^TC)$
for matrix $B,C$. If we define
\begin{align}
	f(U):&=\frac{1}{2}\sum_{i=1}^n\|X_{i\cdot}-U_{i\cdot}\|_2^2=\frac{1}{2}\|X-U\|_F^2\\
	g(U):&=\gamma_2 \sum_{j=1}^p r_j \|u_j\|_2\\
	h(V):&=\gamma_1\sum_{(i,j)\in {\mathcal E}}w_{(i,j)}\|v_{(i,j)}\|_2\ ,
\end{align}
then we have
\begin{align}
	\label{scvxlag}
	L_{\nu} (U,V,\Lambda)=f(U)+g(U)+h(V)+\langle \Lambda,V-A_{{\mathcal E}}U\rangle + \frac{\nu}{2}\|V-A_{{\mathcal E}}U\|_F^2\ ,
\end{align}
and can apply the proposed method.

Then, we consider the proximal gradient map of $h^*$. If we define $r(x):=C\|x\|_2$, then we have
\begin{align}
	r^*(y)=
	\begin{cases}
		0 & \ifq \|y\|_2\leq C\\
		\infty & \other
	\end{cases}\ ,
\end{align}
which means that for $Z=(z_{(i,j),k})_{(i,j)\in {\mathcal E},k=1,\ldots ,p}$, we have
\begin{align}
	\label{scvxhs}
	h^*(Z)=
	\begin{cases}
		0 &  \ifq \|z_{(i,j)}\|_2\leq \gamma_1 w_{(i,j)} \quad \for \quad \forall(i,j)\in {\mathcal E} \\
		\infty & \other
	\end{cases}\ .
\end{align}
Hence, if we map 
$P_C(Z)$ onto 
$C=\{Z\in \Rb^{{\mathcal E}\times p};\|z_{(i,j)}\|_2\leq \gamma_1 w_{(i,j)} \quad \for \quad   (i,j)\in {\mathcal E}\}$
of $Z$, we have
\begin{align*}
	\prox_{\nu h^*}(\Lambda^{(k)} + \nu A_{{\mathcal E}}U^{(k+1)})=P_{C}(\Lambda^{(k)} + \nu A_{{\mathcal E}}U^{(k+1)})\ .
\end{align*}

Finally, we consider the constant $L$ such that $\|\nabla_U \phi_1(U_1)-\nabla_U\phi_1(U_2)\|_F\leq L\|U_1-U_2\|_F$.
From Lemma \ref{nonex}, we have
\begin{align}
	&\quad \|A_\mate^T(\prox_{\nu h^*}(\Lambda + \nu A_{{\mathcal E}}U_1)-\prox_{\nu h^*}(\Lambda + \nu A_{{\mathcal E}}U_2))\|_F \nonumber \\
	&\leq \sqrt{\lambda_{\max}(A_\mate^TA_\mate)} \times \|\prox_{\nu h^*}(\Lambda + \nu A_{{\mathcal E}}U_1)-\prox_{\nu h^*}(\Lambda + \nu A_{{\mathcal E}}U_2)\|_F \nonumber \\
	&\leq \sqrt{\lambda_{\max}(A_\mate^TA_\mate)} \times \|\Lambda + \nu A_{{\mathcal E}}U_1-\Lambda - \nu A_{{\mathcal E}}U_2\|_F \nonumber \\
	&\leq \nu \lambda_{\max}(A_\mate^TA_\mate) \|U_1-U_2\|_F\ .
\end{align}
Since $\nabla_U f(U)=U-X$, we have
\begin{align}
	&\quad\|\nabla_U \phi_1(U_1)-\nabla_U\phi_1(U_2)\|_F \nonumber \\
	&\leq \|\nabla_U f(U_1)-\nabla_U f(U_2)\|_F +\|A_\mate^T(\prox_{\nu h^*}(\Lambda + \nu A_{{\mathcal E}}U_1)-\prox_{\nu h^*}(\Lambda + \nu A_{{\mathcal E}}U_2))\|_F \nonumber \\
	&\leq \|U_1-U_2\| + \nu \lambda_{\max}(A_\mate^TA_\mate) \|U_1-U_2\|_F\ ,
\end{align}
which means that the Lipshitz constant of $\nabla_U \phi_1(U)$ is upperbounded by $1+\nu \lambda_{\max}(A_\mate^TA_\mate)$.
For the derivation of $\lambda_{\max}(A_\mate^TA_\mate)$ and the setting of parameter $\eta$, see Appendix \ref{spap}.

\subsection{Experiments}

We constructed all the programs via Rcpp\footnote{The source code used in the experiments is available at \url{https://github.com/Theveni/SCC_TF}.}.
The AMA is an alternative to the ADMM such that the first step $x^{(k+1)}={\rm argmin}_x L_\nu(x,y^{(k)},\lambda^{(k)})$
is replaced by $x^{(k+1)}={\rm argmin}_x L_0(x,y^{(k)},\lambda^{(k)})$ in Section 2.2.
While the differences between the two algorithms appear to be minor, complexity analysis and numerical experiments show AMA to be significantly more efficient \citep{Chi2015}.

In all experiments, parameter of proposed method is $\eta=\frac{1}{1+\nu_k\max_{i}G_{ii}}$ in Appendix \ref{spap} and $\nu_1 = 1, \nu_{k+1} = 1.1 \nu_k$.Furthermore, the number of features that affect the clusters was set to $p_{true} = 20$.

The data were set to $n=1,000$ and $p = 500$, and 250 data points were generated independently from each of the Gaussian distributions with four different means.
As parameters, $w_{ij}$ used $\phi=\frac{0.5}{p}$, $k = 5$, and $v_i$ was set to 1.
Figure \ref{gamma1} shows the change in calculation time when we fix $\gamma_2 = 10$ and change $\gamma_1 $.
In Figure 1, we can see that the computation time of AMA changes significantly when $\gamma_1$ changes.
In particular, the AMA takes up to 230 seconds when $\gamma_1$ is larger than 5, i.e., when the size of each cluster is large. However, the computation time of the proposed method is stable even when $\gamma_1$ changes. Furthermore, the maximum computation time is only about 10 seconds for all $\gamma_1$, indicating that the computation time can be reduced.

Figure \ref{gamma2} shows the change in the calculation time when we fix $\gamma_1 = 10$ and change $\gamma_2$ for the same data.
We can see that when we change $\gamma_2$, the calculation time of AMA changes greatly depending on the value of $\gamma_2$, similar to the $\gamma_1$ case.
In particular, the AMA takes a long time when $\gamma_2$ is small, i.e., when the result has few zeros and is not sparse, and the maximum time is about 350 seconds. When $\gamma_2$ is large and the solution is sparse, AMA takes less time to compute.
In the proposed method, the fluctuation of the calculation time due to $\gamma_2$ is small, and the calculation time is shorter than that of AMA for all $\gamma_1$.

\begin{figure}[h]
	\begin{tabular}{cc}
		\begin{minipage}{.5\textwidth}
			\begin{center}
				\input{fig_1.tex}
				\caption{The changes in computation time due to $\gamma_1$}
				\label{gamma1}
			\end{center}
		\end{minipage}
		\begin{minipage}{.5\textwidth}
			\begin{center}
				\input{fig_2.tex}
				\caption{The changes in computational time due to $\gamma_2$}
				\label{gamma2}
			\end{center}
		\end{minipage}
	\end{tabular}
\end{figure}

Moreover, we show in Figure \ref{scvx_n} comparison of computation times for propothd method, generalized ADMM, AMA when we fix $p =500, \gamma_1 = 10, \gamma_2 = 10$ and change the number $n$ of data. AMA-FISTA is the calculation time when FISTA is applied to AMA.
The data are generated by $\frac{n}{5}$ from a Gaussian distribution with five different means.
In Figure \ref{scvx_n}, both AMA and AMA-FISTA show a large increase in computation time with respect to the increase in sample size, and the computation time is larger when the sample size is large than the other methods.
The generalized ADMM takes the longest computation time when the sample size is small, but when the sample size becomes large, it can solve the problem more efficiently than AMA and AMA-FISTA.
It can be seen that the proposed method has the smallest increase in computation time with increasing sample size, and the computation time is the shortest for all sample sizes.

Furthermore, we show in Figure \ref{scvx_p} comparison of computation times for propothd method, generalized ADMM, AMA when we fix $n=500,\gamma_1=5,\gamma_2=5$ and
change the number $p$ of variables.
The data were generated by $100$ each from a Gaussian distribution with five different means.
Both AMA and AMA-FISTA have long computation times when the feature dimension is small, but they have the shortest computation time when the feature dimension is large and the solution is sparse. On the other hand, generalized ADMM has a short computation time when the feature dimension is small, but when the feature dimension is large, the computation time is larger than the other methods.
The proposed method has the shortest computation time when the feature dimension is small, and the computation time is almost the same as that of AMA even when the feature dimension is large and sparse, indicating that it can solve the problem efficiently in all cases.

\begin{figure}[h]
	\begin{tabular}{cc}
		\begin{minipage}{.5\textwidth}
			\begin{center}
				\input{fig_3.tex}
				\caption{The changes in computational time due to the number $n$ of variables}
				\label{scvx_n}
			\end{center}
		\end{minipage}
		\begin{minipage}{.5\textwidth}
			\begin{center}
				\input{fig_4.tex}
				\caption{The changes in computational time due to the number $p$ of data}
				\label{scvx_p}
			\end{center}
		\end{minipage}
	\end{tabular}
	\begin{center}
	\end{center}
\end{figure}

\section{Application to Trend Filtering}
\label{tf}

The trend filtering optimization problem is formulated as
\begin{align}
	\label{tfsub}
	\min_x \; \frac{1}{2}\|y-x\|_2^2 + \gamma \|D^{(k+1)}x\|_1 
\end{align}
for an integer $k\geq 0$ and the observed data $y = (y_1, \ldots, y_n)^T \in {\mathbb R}^n$,
where $\gamma \geq 0 $ is the tuning parameter and $D^{(k + 1)}$ is the difference matrix of the order $k+1$ such that
\[
D^{(1)}=\begin{pmatrix}
	-1 & 1 & & &\bigzerol \\
	& -1 & 1 & & & \\
	& & \ddots &\ddots & & \\
	\bigzerou& & & -1 & 1
\end{pmatrix}
\]
for $k=0$, and
\[
D^{(k+1)}=D^{(1)}D^{(k)}\ .
\]

In Figures \ref{k1} and \ref{k2},
we show an example applied to $\sin \theta (0 \leq \theta \leq 2 \pi)$ when $k = 1$ and $k = 2$.
The points are the observation data, and the solid lines are obtained by smoothing via trend filtering.
We observe that the output becomes smoother as the degree $k$ increases.

\begin{figure}[h]
	\begin{tabular}{cc}
		\begin{minipage}{.5\textwidth}
			\begin{center}
				\input{fig_5.tex}
				\caption{Trend filtering with order $k=1$}
				\label{k1}
			\end{center}
		\end{minipage}
		\begin{minipage}{.5\textwidth}
			\begin{center}
				\input{fig_6.tex}
				\caption{Trend filtering with order $k=2$}
				\label{k2}
			\end{center}
		\end{minipage}
	\end{tabular}
	\begin{center}
	\end{center}
\end{figure}

\subsection{Application of the Proposed Method}
To apply the proposed method, we rewrite (\ref{tfsub}) as follows.
\begin{align}
	&\min_{x,y} \frac{1}{2}\|y-x\|_2^2 + \gamma \|z\|_1  \\
	& {\rm subject \quad to} \quad D^{(k+1)}x=z\nonumber
\end{align}
The augmented Lagrangian becomes
\[
L_{\nu}(x,y,\lambda)=\frac{1}{2}\|y-x\|_2^2+\gamma \|z\|_1+\langle \lambda , z-D^{(k+1)}x\rangle + \frac{\nu}{2}\|z-D^{(k+1)}x\|_2^2\ .
\]

If we define 
\begin{align}
	f(x):&=\frac{1}{2}\|y-x\|_2^2 \\
	g(x):&=0\\
	h(z):&=\gamma \|z\|_1\ ,
\end{align}
then we have
\begin{align}
	\label{tflag}
	L_{\nu} (x,z,\lambda)=f(x)+g(x)+h(z)+\langle \lambda , z-D^{(k+1)}x\rangle + \frac{\nu}{2}\|z-D^{(k+1)}x\|_2^2\ .
\end{align}
For this case, we have $\prox_{\eta g}(x-\eta \nabla \phi_1(x))=x-\eta \nabla \phi_1(x)$ due to $g(x)=0$, and
the update of (\ref{teian}) is the standard gradient method rather than the proximal gradient.
The upper bound of the Lipshitz constant in $\nabla \phi_1(x)$
is $1+\nu \lambda_{\max}((D^{(k+1)})^TD^{(k+1)})$, which can be derived from a similar discussion
in Section \ref{sp}.
For the evaluation of $\lambda_{\max}((D^{(k+1)})^TD^{(k+1)})$ and setting of
parameter $\eta>0$, see Appendix \ref{tfap}.

\subsection{Experiments}
We constructed all the programs via Rcpp .
Because the purpose of this paper is to establish the theory of transformation from the ADMM to the proximal gradient
we do not relate comparison with an ADMM procedure proposed in  \citet{trend_admm} that improves performance,
considering an efficient computation of the difference matrix.

In all experiments, parameter of proposed method is $\eta=\frac{1}{1+\nu_k 4^{k+1}}$ in Appendix \ref{tfap} and $\nu_1 = 1, \nu_{k+1} = 1.1 \nu_k$.

We generate $n = 1,000$ data by adding noise to $\sin \theta$, as shown in Figures \ref{k1}, \ref{k2}.
Figure \ref{tf_g_k1}, \ref{tf_g_k2} shows the change in calculation time when the value of $\gamma$ is changed with respect to $k = 1,2$.
For $k=1$, the computation time increases as $\gamma$ increases for both ADMM and the proposed method. The computation time of the proposed method is shorter than that of ADMM for all $\gamma$, and for large $\gamma$, i.e., the computation time is about $\frac{1}{4}$ in the sparse case where $D^{(k)}\beta$ of the solution $\beta$ has many $0$.

In the case of $k=2$, as in the case of $k=1$, the computation time of both methods increases as $\gamma$ increases. When $\gamma$ is small, the ADMM and the proposed method have similar computation times, but when $\gamma$ is large, the computation time is $\frac{1}{3}$. The results show that the proposed method is more efficient than ADMM in both cases of $k=1,k=2$.

\begin{figure}[h]
	\begin{tabular}{cc}
		\begin{minipage}{.5\textwidth}
			\begin{center}
				\input{fig_7.tex}
				\caption{The changes in computational time due to $\gamma$ when $k=1$}
				\label{tf_g_k1}
			\end{center}
		\end{minipage}
		\begin{minipage}{.5\textwidth}
			\begin{center}
				\input{fig_8.tex}
				\caption{The changes in computational time due to $\gamma$ when $k=2$}
				\label{tf_g_k2}
			\end{center}
		\end{minipage}
	\end{tabular}
	\begin{center}
	\end{center}
\end{figure}

\section{Conclusion}
\label{conclusion}

In this paper, we proposed a general method to convert the solution of the optimization problem by ADMM to the solution using the proximal gradient method.
In addition, numerical experiments showed that it can be applied to sparse estimation problems such as sparse convex clustering and trend filtering,
resulting in significant efficiency improvements.In particular, for both sparse convex clustering and trend filtering, the proposed method is much more efficient than existing methods such as ADMM when the regularization parameter is large such that the results are sparse. This suggests that the proposed method can perform efficient computation by making good use of the sparsity that the result becomes zero.

In applying the proposed method,
it is premised that a Lipschtz constant or an upper bound is obtained.
This method is expected to apply not only to existing sparse estimation problems but also to many problems of adding two regularization terms to the loss function.
In that case, the problem of finding an efficient solution is reduced to the problem of finding the Lipschtz coefficient.

In this study, we focus on sparse estimation and its surrounding problems,
however, it is necessary to actively apply it to optimization problems in general and further clarify its effectiveness.

\appendix
	
	\section{The Setting of Proximal Gradient Parameter $\eta$}
	\subsection{Sparse Convex Clustering}
	\label{spap}
	Let
	$A_\mate =(a_{(i,j),k})_{(i,j)\in \mate,k=1\ldots n}\in \Rb^{\mate\times n}$. Then, we have
	\begin{align}
		\label{Ae}
		a_{(i,j),k}=
		\begin{cases}
			1 & \ifq k=i \\
			-1 & \ifq  k=j \\
			0& \other 
		\end{cases}\ ,
	\end{align}
	and the $(i,j)$ element of $G$ for 
	$A_\mate^T A_\mate=:G\in \Rb^{n\times n}$ can be written as
	\begin{align}
		\label{scvxeigen}
		G_{ij}=\begin{cases}
			-1 & \ifq (i,j)\in \mate\\
			\sum_{k\neq i}^n \lvert G_{ik} \rvert & \ifq i=j \\
			0 & \other 
		\end{cases}\ .
	\end{align}
	Then, we notice the following lemma:
	\begin{hodai}[Gershgorin]
		\label{gershgorin}
		Assume we have symmetric matrix $A\in \Rb^{n\times n}$.
		\begin{align}
			\label{gersh}
			\lambda_{\max}(A)\leq \max_{i=1,\ldots , n}(a_{ii}+\sum_{j\neq i}^n\lvert a_{ij} \rvert)
		\end{align}
	\end{hodai}
	From Lemma \ref{gershgorin} because of
	\begin{align}
		\lambda_{\max}(A_\mate^TA_\mate)\leq 2\max_{i=1,\ldots,n}G_{ii}\, 
	\end{align}
	it is appropriate to set $\eta>0$ as
	\begin{align}
		\label{scvxpara}
		\eta = \frac{1}{1+2\nu\max_{i=1,\ldots,n}G_{ii}}\ .
	\end{align}
	
	\subsection{Trend Filtering}
	\label{tfap}
	For $k\geq 0$, $D^{(k+1)}\in \Rb^{(n-k)\times n}$ can be written as
	\footnotesize
	\[
	D^{(k+1)}=
	\begin{pmatrix}
		(-1)^{k+1}{}_{k+1}\Cm_0 & (-1)^{k+2}{}_{k+1}\Cm_1 & \cdots & _{k+1}\Cm_{k+1}&  &    &  \bigzerol \\
		&(-1)^{k+1}{}_{k+1}\Cm_0 & (-1)^{k+2}{}_{k+1}\Cm_1 & \cdots & {}_{k+1}\Cm_{k+1}& & & \\
		& & \ddots &\ddots & \ddots & \ddots & & &\\
		\bigzerou& & &(-1)^{k+1}{}_{k+1}\Cm_0 & (-1)^{k+2}{}_{k+1}\Cm_1 & \cdots & {}_{k+1}\Cm_{k+1}
	\end{pmatrix}\ ,
	\]
	\normalsize
	i.e., the $(i,j)$ element of $D^{(k+1)}$ is 
	\begin{align}
		D_{ij}^{(k+1)}=
		\begin{cases}
			(-1)^{k+1+j-i}{}_{k+1}\Cm_{j-i} & \ifq 0\leq j-i \leq k+1\\
			0 &\other
		\end{cases}\ .
	\end{align}
	Thus, from Lemma \ref{gershgorin}, we have
	\begin{align}
		\lambda_{\max}\left((D^{(k+1)})^TD^{(k+1)}\right) &\leq \max_{i=1,\ldots ,n} \sum_{j=1}^n \lvert \left((D^{(k+1)})^TD^{(k+1)}\right)_{ij}\rvert \nonumber \\
		&=\max_{i=1,\ldots,n}\lvert \sum_{j=1}^n\sum_{s=1}^n D_{si}D_{sj} \rvert \nonumber \\
		&\leq \max_{i=1,\ldots,n}\sum_{j=1}^n\sum_{s=1}^n \lvert D_{si}D_{sj} \rvert \nonumber \\
		&\leq \sum_{j=0}^{k+1}\sum_{s=0}^{k+1} {}_{k+1}\Cm_j \times  {}_{k+1}\Cm_s \nonumber \\
		& = (\sum_{j=0}^{k+1} {}_{k+1}\Cm_j)^2 = 4^{k+1}\ ,
	\end{align}
	and it is appropriate to set $\eta>0$ as 
	\[
	\eta = \frac{1}{1+\nu 4^{k+1}}\ .
	\]
	
	
	



\input{shimmura.bbl}


\end{document}

%% file: fig_1.tex
\begin{tikzpicture}[x=0.8pt,y=0.8pt]
\definecolor{fillColor}{RGB}{255,255,255}
\path[use as bounding box,fill=fillColor,fill opacity=0.00] (0,0) rectangle (216.81,216.81);
\begin{scope}
\path[clip] (  0.00,  0.00) rectangle (216.81,216.81);
\definecolor{drawColor}{RGB}{255,255,255}
\definecolor{fillColor}{RGB}{255,255,255}

\path[draw=drawColor,line width= 0.6pt,line join=round,line cap=round,fill=fillColor] (  0.00,  0.00) rectangle (216.81,216.81);
\end{scope}
\begin{scope}
\path[clip] ( 43.58, 36.18) rectangle (211.31,211.31);
\definecolor{fillColor}{gray}{0.92}

\path[fill=fillColor] ( 43.58, 36.18) rectangle (211.31,211.31);
\definecolor{drawColor}{RGB}{255,255,255}

\path[draw=drawColor,line width= 0.3pt,line join=round] ( 43.58, 61.48) --
	(211.31, 61.48);

\path[draw=drawColor,line width= 0.3pt,line join=round] ( 43.58, 96.15) --
	(211.31, 96.15);

\path[draw=drawColor,line width= 0.3pt,line join=round] ( 43.58,130.82) --
	(211.31,130.82);

\path[draw=drawColor,line width= 0.3pt,line join=round] ( 43.58,165.49) --
	(211.31,165.49);

\path[draw=drawColor,line width= 0.3pt,line join=round] ( 43.58,200.16) --
	(211.31,200.16);

\path[draw=drawColor,line width= 0.3pt,line join=round] ( 70.26, 36.18) --
	( 70.26,211.31);

\path[draw=drawColor,line width= 0.3pt,line join=round] (108.38, 36.18) --
	(108.38,211.31);

\path[draw=drawColor,line width= 0.3pt,line join=round] (146.50, 36.18) --
	(146.50,211.31);

\path[draw=drawColor,line width= 0.3pt,line join=round] (184.63, 36.18) --
	(184.63,211.31);

\path[draw=drawColor,line width= 0.6pt,line join=round] ( 43.58, 44.14) --
	(211.31, 44.14);

\path[draw=drawColor,line width= 0.6pt,line join=round] ( 43.58, 78.81) --
	(211.31, 78.81);

\path[draw=drawColor,line width= 0.6pt,line join=round] ( 43.58,113.48) --
	(211.31,113.48);

\path[draw=drawColor,line width= 0.6pt,line join=round] ( 43.58,148.15) --
	(211.31,148.15);

\path[draw=drawColor,line width= 0.6pt,line join=round] ( 43.58,182.82) --
	(211.31,182.82);

\path[draw=drawColor,line width= 0.6pt,line join=round] ( 51.20, 36.18) --
	( 51.20,211.31);

\path[draw=drawColor,line width= 0.6pt,line join=round] ( 89.32, 36.18) --
	( 89.32,211.31);

\path[draw=drawColor,line width= 0.6pt,line join=round] (127.44, 36.18) --
	(127.44,211.31);

\path[draw=drawColor,line width= 0.6pt,line join=round] (165.57, 36.18) --
	(165.57,211.31);

\path[draw=drawColor,line width= 0.6pt,line join=round] (203.69, 36.18) --
	(203.69,211.31);
\definecolor{drawColor}{RGB}{0,191,196}

\path[draw=drawColor,line width= 0.6pt,line join=round] ( 51.20, 45.25) --
	( 66.45, 45.81) --
	( 81.70, 45.53) --
	( 96.95, 45.81) --
	(112.20, 46.08) --
	(127.44, 45.53) --
	(142.69, 45.81) --
	(157.94, 46.36) --
	(173.19, 46.91) --
	(188.44, 49.41) --
	(203.69, 49.41);

\path[draw=drawColor,line width= 0.4pt,line join=round,line cap=round] ( 51.20, 45.25) circle (  1.96);

\path[draw=drawColor,line width= 0.4pt,line join=round,line cap=round] ( 66.45, 45.81) circle (  1.96);

\path[draw=drawColor,line width= 0.4pt,line join=round,line cap=round] ( 81.70, 45.53) circle (  1.96);

\path[draw=drawColor,line width= 0.4pt,line join=round,line cap=round] ( 96.95, 45.81) circle (  1.96);

\path[draw=drawColor,line width= 0.4pt,line join=round,line cap=round] (112.20, 46.08) circle (  1.96);

\path[draw=drawColor,line width= 0.4pt,line join=round,line cap=round] (127.44, 45.53) circle (  1.96);

\path[draw=drawColor,line width= 0.4pt,line join=round,line cap=round] (142.69, 45.81) circle (  1.96);

\path[draw=drawColor,line width= 0.4pt,line join=round,line cap=round] (157.94, 46.36) circle (  1.96);

\path[draw=drawColor,line width= 0.4pt,line join=round,line cap=round] (173.19, 46.91) circle (  1.96);

\path[draw=drawColor,line width= 0.4pt,line join=round,line cap=round] (188.44, 49.41) circle (  1.96);

\path[draw=drawColor,line width= 0.4pt,line join=round,line cap=round] (203.69, 49.41) circle (  1.96);
\definecolor{drawColor}{RGB}{248,118,109}

\path[draw=drawColor,line width= 0.6pt,line join=round] ( 51.20, 44.14) --
	( 66.45, 46.36) --
	( 81.70, 49.97) --
	( 96.95, 55.51) --
	(112.20, 64.67) --
	(127.44, 78.53) --
	(142.69,104.88) --
	(157.94,203.35) --
	(173.19,110.15) --
	(188.44, 72.15) --
	(203.69, 71.32);

\path[draw=drawColor,line width= 0.4pt,line join=round,line cap=round] ( 51.20, 44.14) circle (  1.96);

\path[draw=drawColor,line width= 0.4pt,line join=round,line cap=round] ( 66.45, 46.36) circle (  1.96);

\path[draw=drawColor,line width= 0.4pt,line join=round,line cap=round] ( 81.70, 49.97) circle (  1.96);

\path[draw=drawColor,line width= 0.4pt,line join=round,line cap=round] ( 96.95, 55.51) circle (  1.96);

\path[draw=drawColor,line width= 0.4pt,line join=round,line cap=round] (112.20, 64.67) circle (  1.96);

\path[draw=drawColor,line width= 0.4pt,line join=round,line cap=round] (127.44, 78.53) circle (  1.96);

\path[draw=drawColor,line width= 0.4pt,line join=round,line cap=round] (142.69,104.88) circle (  1.96);

\path[draw=drawColor,line width= 0.4pt,line join=round,line cap=round] (157.94,203.35) circle (  1.96);

\path[draw=drawColor,line width= 0.4pt,line join=round,line cap=round] (173.19,110.15) circle (  1.96);

\path[draw=drawColor,line width= 0.4pt,line join=round,line cap=round] (188.44, 72.15) circle (  1.96);

\path[draw=drawColor,line width= 0.4pt,line join=round,line cap=round] (203.69, 71.32) circle (  1.96);
\end{scope}
\begin{scope}
\path[clip] (  0.00,  0.00) rectangle (216.81,216.81);
\definecolor{drawColor}{gray}{0.30}

\node[text=drawColor,anchor=base east,inner sep=0pt, outer sep=0pt, scale=  1.20] at ( 38.63, 40.01) {0};

\node[text=drawColor,anchor=base east,inner sep=0pt, outer sep=0pt, scale=  1.20] at ( 38.63, 74.68) {50};

\node[text=drawColor,anchor=base east,inner sep=0pt, outer sep=0pt, scale=  1.20] at ( 38.63,109.35) {100};

\node[text=drawColor,anchor=base east,inner sep=0pt, outer sep=0pt, scale=  1.20] at ( 38.63,144.02) {150};

\node[text=drawColor,anchor=base east,inner sep=0pt, outer sep=0pt, scale=  1.20] at ( 38.63,178.69) {200};
\end{scope}
\begin{scope}
\path[clip] (  0.00,  0.00) rectangle (216.81,216.81);
\definecolor{drawColor}{gray}{0.20}

\path[draw=drawColor,line width= 0.6pt,line join=round] ( 40.83, 44.14) --
	( 43.58, 44.14);

\path[draw=drawColor,line width= 0.6pt,line join=round] ( 40.83, 78.81) --
	( 43.58, 78.81);

\path[draw=drawColor,line width= 0.6pt,line join=round] ( 40.83,113.48) --
	( 43.58,113.48);

\path[draw=drawColor,line width= 0.6pt,line join=round] ( 40.83,148.15) --
	( 43.58,148.15);

\path[draw=drawColor,line width= 0.6pt,line join=round] ( 40.83,182.82) --
	( 43.58,182.82);
\end{scope}
\begin{scope}
\path[clip] (  0.00,  0.00) rectangle (216.81,216.81);
\definecolor{drawColor}{gray}{0.20}

\path[draw=drawColor,line width= 0.6pt,line join=round] ( 51.20, 33.43) --
	( 51.20, 36.18);

\path[draw=drawColor,line width= 0.6pt,line join=round] ( 89.32, 33.43) --
	( 89.32, 36.18);

\path[draw=drawColor,line width= 0.6pt,line join=round] (127.44, 33.43) --
	(127.44, 36.18);

\path[draw=drawColor,line width= 0.6pt,line join=round] (165.57, 33.43) --
	(165.57, 36.18);

\path[draw=drawColor,line width= 0.6pt,line join=round] (203.69, 33.43) --
	(203.69, 36.18);
\end{scope}
\begin{scope}
\path[clip] (  0.00,  0.00) rectangle (216.81,216.81);
\definecolor{drawColor}{gray}{0.30}

\node[text=drawColor,anchor=base,inner sep=0pt, outer sep=0pt, scale=  1.20] at ( 51.20, 22.97) {0.0};

\node[text=drawColor,anchor=base,inner sep=0pt, outer sep=0pt, scale=  1.20] at ( 89.32, 22.97) {2.5};

\node[text=drawColor,anchor=base,inner sep=0pt, outer sep=0pt, scale=  1.20] at (127.44, 22.97) {5.0};

\node[text=drawColor,anchor=base,inner sep=0pt, outer sep=0pt, scale=  1.20] at (165.57, 22.97) {7.5};

\node[text=drawColor,anchor=base,inner sep=0pt, outer sep=0pt, scale=  1.20] at (203.69, 22.97) {10.0};
\end{scope}
\begin{scope}
\path[clip] (  0.00,  0.00) rectangle (216.81,216.81);
\definecolor{drawColor}{RGB}{0,0,0}

\node[text=drawColor,anchor=base,inner sep=0pt, outer sep=0pt, scale=  1.40] at (127.44,  8.22) {\bfseries $\gamma_1$};
\end{scope}
\begin{scope}
\path[clip] (  0.00,  0.00) rectangle (216.81,216.81);
\definecolor{drawColor}{RGB}{0,0,0}

\node[text=drawColor,rotate= 90.00,anchor=base,inner sep=0pt, outer sep=0pt, scale=  1.40] at ( 15.16,123.75) {\bfseries time[s]};
\end{scope}
\begin{scope}
\path[clip] (  0.00,  0.00) rectangle (216.81,216.81);
\definecolor{fillColor}{RGB}{255,255,255}

\path[fill=fillColor] ( 43.58,165.90) rectangle (123.12,211.31);
\end{scope}
\begin{scope}
\path[clip] (  0.00,  0.00) rectangle (216.81,216.81);
\definecolor{fillColor}{gray}{0.95}

\path[fill=fillColor] ( 49.08,185.86) rectangle ( 63.53,200.31);
\end{scope}
\begin{scope}
\path[clip] (  0.00,  0.00) rectangle (216.81,216.81);
\definecolor{drawColor}{RGB}{248,118,109}

\path[draw=drawColor,line width= 0.6pt,line join=round] ( 50.52,193.08) -- ( 62.09,193.08);
\end{scope}
\begin{scope}
\path[clip] (  0.00,  0.00) rectangle (216.81,216.81);
\definecolor{drawColor}{RGB}{248,118,109}

\path[draw=drawColor,line width= 0.4pt,line join=round,line cap=round] ( 56.31,193.08) circle (  1.96);
\end{scope}
\begin{scope}
\path[clip] (  0.00,  0.00) rectangle (216.81,216.81);
\definecolor{drawColor}{RGB}{248,118,109}

\path[draw=drawColor,line width= 0.6pt,line join=round] ( 50.52,193.08) -- ( 62.09,193.08);
\end{scope}
\begin{scope}
\path[clip] (  0.00,  0.00) rectangle (216.81,216.81);
\definecolor{drawColor}{RGB}{248,118,109}

\path[draw=drawColor,line width= 0.4pt,line join=round,line cap=round] ( 56.31,193.08) circle (  1.96);
\end{scope}
\begin{scope}
\path[clip] (  0.00,  0.00) rectangle (216.81,216.81);
\definecolor{fillColor}{gray}{0.95}

\path[fill=fillColor] ( 49.08,171.40) rectangle ( 63.53,185.86);
\end{scope}
\begin{scope}
\path[clip] (  0.00,  0.00) rectangle (216.81,216.81);
\definecolor{drawColor}{RGB}{0,191,196}

\path[draw=drawColor,line width= 0.6pt,line join=round] ( 50.52,178.63) -- ( 62.09,178.63);
\end{scope}
\begin{scope}
\path[clip] (  0.00,  0.00) rectangle (216.81,216.81);
\definecolor{drawColor}{RGB}{0,191,196}

\path[draw=drawColor,line width= 0.4pt,line join=round,line cap=round] ( 56.31,178.63) circle (  1.96);
\end{scope}
\begin{scope}
\path[clip] (  0.00,  0.00) rectangle (216.81,216.81);
\definecolor{drawColor}{RGB}{0,191,196}

\path[draw=drawColor,line width= 0.6pt,line join=round] ( 50.52,178.63) -- ( 62.09,178.63);
\end{scope}
\begin{scope}
\path[clip] (  0.00,  0.00) rectangle (216.81,216.81);
\definecolor{drawColor}{RGB}{0,191,196}

\path[draw=drawColor,line width= 0.4pt,line join=round,line cap=round] ( 56.31,178.63) circle (  1.96);
\end{scope}
\begin{scope}
\path[clip] (  0.00,  0.00) rectangle (216.81,216.81);
\definecolor{drawColor}{RGB}{0,0,0}

\node[text=drawColor,anchor=base west,inner sep=0pt, outer sep=0pt, scale=  1.20] at ( 69.03,188.95) {AMA};
\end{scope}
\begin{scope}
\path[clip] (  0.00,  0.00) rectangle (216.81,216.81);
\definecolor{drawColor}{RGB}{0,0,0}

\node[text=drawColor,anchor=base west,inner sep=0pt, outer sep=0pt, scale=  1.20] at ( 69.03,174.50) {Proposed};
\end{scope}
\end{tikzpicture}

%% file: fig_2.tex
\begin{tikzpicture}[x=0.8pt,y=0.8pt]
\definecolor{fillColor}{RGB}{255,255,255}
\path[use as bounding box,fill=fillColor,fill opacity=0.00] (0,0) rectangle (216.81,216.81);
\begin{scope}
\path[clip] (  0.00,  0.00) rectangle (216.81,216.81);
\definecolor{drawColor}{RGB}{255,255,255}
\definecolor{fillColor}{RGB}{255,255,255}

\path[draw=drawColor,line width= 0.6pt,line join=round,line cap=round,fill=fillColor] (  0.00,  0.00) rectangle (216.81,216.81);
\end{scope}
\begin{scope}
\path[clip] ( 43.58, 36.18) rectangle (211.31,211.31);
\definecolor{fillColor}{gray}{0.92}

\path[fill=fillColor] ( 43.58, 36.18) rectangle (211.31,211.31);
\definecolor{drawColor}{RGB}{255,255,255}

\path[draw=drawColor,line width= 0.3pt,line join=round] ( 43.58, 66.61) --
	(211.31, 66.61);

\path[draw=drawColor,line width= 0.3pt,line join=round] ( 43.58,112.65) --
	(211.31,112.65);

\path[draw=drawColor,line width= 0.3pt,line join=round] ( 43.58,158.69) --
	(211.31,158.69);

\path[draw=drawColor,line width= 0.3pt,line join=round] ( 43.58,204.73) --
	(211.31,204.73);

\path[draw=drawColor,line width= 0.3pt,line join=round] ( 70.26, 36.18) --
	( 70.26,211.31);

\path[draw=drawColor,line width= 0.3pt,line join=round] (108.38, 36.18) --
	(108.38,211.31);

\path[draw=drawColor,line width= 0.3pt,line join=round] (146.50, 36.18) --
	(146.50,211.31);

\path[draw=drawColor,line width= 0.3pt,line join=round] (184.63, 36.18) --
	(184.63,211.31);

\path[draw=drawColor,line width= 0.6pt,line join=round] ( 43.58, 43.59) --
	(211.31, 43.59);

\path[draw=drawColor,line width= 0.6pt,line join=round] ( 43.58, 89.63) --
	(211.31, 89.63);

\path[draw=drawColor,line width= 0.6pt,line join=round] ( 43.58,135.67) --
	(211.31,135.67);

\path[draw=drawColor,line width= 0.6pt,line join=round] ( 43.58,181.71) --
	(211.31,181.71);

\path[draw=drawColor,line width= 0.6pt,line join=round] ( 51.20, 36.18) --
	( 51.20,211.31);

\path[draw=drawColor,line width= 0.6pt,line join=round] ( 89.32, 36.18) --
	( 89.32,211.31);

\path[draw=drawColor,line width= 0.6pt,line join=round] (127.44, 36.18) --
	(127.44,211.31);

\path[draw=drawColor,line width= 0.6pt,line join=round] (165.57, 36.18) --
	(165.57,211.31);

\path[draw=drawColor,line width= 0.6pt,line join=round] (203.69, 36.18) --
	(203.69,211.31);
\definecolor{drawColor}{RGB}{0,191,196}

\path[draw=drawColor,line width= 0.6pt,line join=round] ( 51.20, 44.14) --
	( 66.45, 44.42) --
	( 81.70, 44.46) --
	( 96.95, 44.14) --
	(112.20, 45.25) --
	(127.44, 44.42) --
	(142.69, 44.69) --
	(157.94, 44.42) --
	(173.19, 44.37) --
	(188.44, 44.65) --
	(203.69, 44.42);

\path[draw=drawColor,line width= 0.4pt,line join=round,line cap=round] ( 51.20, 44.14) circle (  1.96);

\path[draw=drawColor,line width= 0.4pt,line join=round,line cap=round] ( 66.45, 44.42) circle (  1.96);

\path[draw=drawColor,line width= 0.4pt,line join=round,line cap=round] ( 81.70, 44.46) circle (  1.96);

\path[draw=drawColor,line width= 0.4pt,line join=round,line cap=round] ( 96.95, 44.14) circle (  1.96);

\path[draw=drawColor,line width= 0.4pt,line join=round,line cap=round] (112.20, 45.25) circle (  1.96);

\path[draw=drawColor,line width= 0.4pt,line join=round,line cap=round] (127.44, 44.42) circle (  1.96);

\path[draw=drawColor,line width= 0.4pt,line join=round,line cap=round] (142.69, 44.69) circle (  1.96);

\path[draw=drawColor,line width= 0.4pt,line join=round,line cap=round] (157.94, 44.42) circle (  1.96);

\path[draw=drawColor,line width= 0.4pt,line join=round,line cap=round] (173.19, 44.37) circle (  1.96);

\path[draw=drawColor,line width= 0.4pt,line join=round,line cap=round] (188.44, 44.65) circle (  1.96);

\path[draw=drawColor,line width= 0.4pt,line join=round,line cap=round] (203.69, 44.42) circle (  1.96);
\definecolor{drawColor}{RGB}{248,118,109}

\path[draw=drawColor,line width= 0.6pt,line join=round] ( 51.20, 64.89) --
	( 66.45, 71.44) --
	( 81.70, 97.69) --
	( 96.95,203.35) --
	(112.20, 53.16) --
	(127.44, 49.34) --
	(142.69, 47.69) --
	(157.94, 47.41) --
	(173.19, 47.13) --
	(188.44, 46.86) --
	(203.69, 46.87);

\path[draw=drawColor,line width= 0.4pt,line join=round,line cap=round] ( 51.20, 64.89) circle (  1.96);

\path[draw=drawColor,line width= 0.4pt,line join=round,line cap=round] ( 66.45, 71.44) circle (  1.96);

\path[draw=drawColor,line width= 0.4pt,line join=round,line cap=round] ( 81.70, 97.69) circle (  1.96);

\path[draw=drawColor,line width= 0.4pt,line join=round,line cap=round] ( 96.95,203.35) circle (  1.96);

\path[draw=drawColor,line width= 0.4pt,line join=round,line cap=round] (112.20, 53.16) circle (  1.96);

\path[draw=drawColor,line width= 0.4pt,line join=round,line cap=round] (127.44, 49.34) circle (  1.96);

\path[draw=drawColor,line width= 0.4pt,line join=round,line cap=round] (142.69, 47.69) circle (  1.96);

\path[draw=drawColor,line width= 0.4pt,line join=round,line cap=round] (157.94, 47.41) circle (  1.96);

\path[draw=drawColor,line width= 0.4pt,line join=round,line cap=round] (173.19, 47.13) circle (  1.96);

\path[draw=drawColor,line width= 0.4pt,line join=round,line cap=round] (188.44, 46.86) circle (  1.96);

\path[draw=drawColor,line width= 0.4pt,line join=round,line cap=round] (203.69, 46.87) circle (  1.96);
\end{scope}
\begin{scope}
\path[clip] (  0.00,  0.00) rectangle (216.81,216.81);
\definecolor{drawColor}{gray}{0.30}

\node[text=drawColor,anchor=base east,inner sep=0pt, outer sep=0pt, scale=  1.20] at ( 38.63, 39.46) {0};

\node[text=drawColor,anchor=base east,inner sep=0pt, outer sep=0pt, scale=  1.20] at ( 38.63, 85.50) {100};

\node[text=drawColor,anchor=base east,inner sep=0pt, outer sep=0pt, scale=  1.20] at ( 38.63,131.54) {200};

\node[text=drawColor,anchor=base east,inner sep=0pt, outer sep=0pt, scale=  1.20] at ( 38.63,177.58) {300};
\end{scope}
\begin{scope}
\path[clip] (  0.00,  0.00) rectangle (216.81,216.81);
\definecolor{drawColor}{gray}{0.20}

\path[draw=drawColor,line width= 0.6pt,line join=round] ( 40.83, 43.59) --
	( 43.58, 43.59);

\path[draw=drawColor,line width= 0.6pt,line join=round] ( 40.83, 89.63) --
	( 43.58, 89.63);

\path[draw=drawColor,line width= 0.6pt,line join=round] ( 40.83,135.67) --
	( 43.58,135.67);

\path[draw=drawColor,line width= 0.6pt,line join=round] ( 40.83,181.71) --
	( 43.58,181.71);
\end{scope}
\begin{scope}
\path[clip] (  0.00,  0.00) rectangle (216.81,216.81);
\definecolor{drawColor}{gray}{0.20}

\path[draw=drawColor,line width= 0.6pt,line join=round] ( 51.20, 33.43) --
	( 51.20, 36.18);

\path[draw=drawColor,line width= 0.6pt,line join=round] ( 89.32, 33.43) --
	( 89.32, 36.18);

\path[draw=drawColor,line width= 0.6pt,line join=round] (127.44, 33.43) --
	(127.44, 36.18);

\path[draw=drawColor,line width= 0.6pt,line join=round] (165.57, 33.43) --
	(165.57, 36.18);

\path[draw=drawColor,line width= 0.6pt,line join=round] (203.69, 33.43) --
	(203.69, 36.18);
\end{scope}
\begin{scope}
\path[clip] (  0.00,  0.00) rectangle (216.81,216.81);
\definecolor{drawColor}{gray}{0.30}

\node[text=drawColor,anchor=base,inner sep=0pt, outer sep=0pt, scale=  1.20] at ( 51.20, 22.97) {0.0};

\node[text=drawColor,anchor=base,inner sep=0pt, outer sep=0pt, scale=  1.20] at ( 89.32, 22.97) {2.5};

\node[text=drawColor,anchor=base,inner sep=0pt, outer sep=0pt, scale=  1.20] at (127.44, 22.97) {5.0};

\node[text=drawColor,anchor=base,inner sep=0pt, outer sep=0pt, scale=  1.20] at (165.57, 22.97) {7.5};

\node[text=drawColor,anchor=base,inner sep=0pt, outer sep=0pt, scale=  1.20] at (203.69, 22.97) {10.0};
\end{scope}
\begin{scope}
\path[clip] (  0.00,  0.00) rectangle (216.81,216.81);
\definecolor{drawColor}{RGB}{0,0,0}

\node[text=drawColor,anchor=base,inner sep=0pt, outer sep=0pt, scale=  1.40] at (127.44,  8.22) {\bfseries $\gamma_2$};
\end{scope}
\begin{scope}
\path[clip] (  0.00,  0.00) rectangle (216.81,216.81);
\definecolor{drawColor}{RGB}{0,0,0}

\node[text=drawColor,rotate= 90.00,anchor=base,inner sep=0pt, outer sep=0pt, scale=  1.40] at ( 15.16,123.75) {\bfseries time[s]};
\end{scope}
\begin{scope}
\path[clip] (  0.00,  0.00) rectangle (216.81,216.81);
\definecolor{fillColor}{RGB}{255,255,255}

\path[fill=fillColor] (131.77,165.90) rectangle (211.31,211.31);
\end{scope}
\begin{scope}
\path[clip] (  0.00,  0.00) rectangle (216.81,216.81);
\definecolor{fillColor}{gray}{0.95}

\path[fill=fillColor] (137.27,185.86) rectangle (151.72,200.31);
\end{scope}
\begin{scope}
\path[clip] (  0.00,  0.00) rectangle (216.81,216.81);
\definecolor{drawColor}{RGB}{248,118,109}

\path[draw=drawColor,line width= 0.6pt,line join=round] (138.71,193.08) -- (150.28,193.08);
\end{scope}
\begin{scope}
\path[clip] (  0.00,  0.00) rectangle (216.81,216.81);
\definecolor{drawColor}{RGB}{248,118,109}

\path[draw=drawColor,line width= 0.4pt,line join=round,line cap=round] (144.49,193.08) circle (  1.96);
\end{scope}
\begin{scope}
\path[clip] (  0.00,  0.00) rectangle (216.81,216.81);
\definecolor{drawColor}{RGB}{248,118,109}

\path[draw=drawColor,line width= 0.6pt,line join=round] (138.71,193.08) -- (150.28,193.08);
\end{scope}
\begin{scope}
\path[clip] (  0.00,  0.00) rectangle (216.81,216.81);
\definecolor{drawColor}{RGB}{248,118,109}

\path[draw=drawColor,line width= 0.4pt,line join=round,line cap=round] (144.49,193.08) circle (  1.96);
\end{scope}
\begin{scope}
\path[clip] (  0.00,  0.00) rectangle (216.81,216.81);
\definecolor{fillColor}{gray}{0.95}

\path[fill=fillColor] (137.27,171.40) rectangle (151.72,185.86);
\end{scope}
\begin{scope}
\path[clip] (  0.00,  0.00) rectangle (216.81,216.81);
\definecolor{drawColor}{RGB}{0,191,196}

\path[draw=drawColor,line width= 0.6pt,line join=round] (138.71,178.63) -- (150.28,178.63);
\end{scope}
\begin{scope}
\path[clip] (  0.00,  0.00) rectangle (216.81,216.81);
\definecolor{drawColor}{RGB}{0,191,196}

\path[draw=drawColor,line width= 0.4pt,line join=round,line cap=round] (144.49,178.63) circle (  1.96);
\end{scope}
\begin{scope}
\path[clip] (  0.00,  0.00) rectangle (216.81,216.81);
\definecolor{drawColor}{RGB}{0,191,196}

\path[draw=drawColor,line width= 0.6pt,line join=round] (138.71,178.63) -- (150.28,178.63);
\end{scope}
\begin{scope}
\path[clip] (  0.00,  0.00) rectangle (216.81,216.81);
\definecolor{drawColor}{RGB}{0,191,196}

\path[draw=drawColor,line width= 0.4pt,line join=round,line cap=round] (144.49,178.63) circle (  1.96);
\end{scope}
\begin{scope}
\path[clip] (  0.00,  0.00) rectangle (216.81,216.81);
\definecolor{drawColor}{RGB}{0,0,0}

\node[text=drawColor,anchor=base west,inner sep=0pt, outer sep=0pt, scale=  1.20] at (157.22,188.95) {AMA};
\end{scope}
\begin{scope}
\path[clip] (  0.00,  0.00) rectangle (216.81,216.81);
\definecolor{drawColor}{RGB}{0,0,0}

\node[text=drawColor,anchor=base west,inner sep=0pt, outer sep=0pt, scale=  1.20] at (157.22,174.50) {Proposed};
\end{scope}
\end{tikzpicture}

%% file: fig_3.tex
\begin{tikzpicture}[x=0.8pt,y=0.8pt]
\definecolor{fillColor}{RGB}{255,255,255}
\path[use as bounding box,fill=fillColor,fill opacity=0.00] (0,0) rectangle (216.81,216.81);
\begin{scope}
\path[clip] (  0.00,  0.00) rectangle (216.81,216.81);
\definecolor{drawColor}{RGB}{255,255,255}
\definecolor{fillColor}{RGB}{255,255,255}

\path[draw=drawColor,line width= 0.6pt,line join=round,line cap=round,fill=fillColor] (  0.00,  0.00) rectangle (216.81,216.81);
\end{scope}
\begin{scope}
\path[clip] ( 49.58, 36.18) rectangle (211.31,211.31);
\definecolor{fillColor}{gray}{0.92}

\path[fill=fillColor] ( 49.58, 36.18) rectangle (211.31,211.31);
\definecolor{drawColor}{RGB}{255,255,255}

\path[draw=drawColor,line width= 0.3pt,line join=round] ( 49.58, 52.04) --
	(211.31, 52.04);

\path[draw=drawColor,line width= 0.3pt,line join=round] ( 49.58, 87.73) --
	(211.31, 87.73);

\path[draw=drawColor,line width= 0.3pt,line join=round] ( 49.58,123.42) --
	(211.31,123.42);

\path[draw=drawColor,line width= 0.3pt,line join=round] ( 49.58,159.10) --
	(211.31,159.10);

\path[draw=drawColor,line width= 0.3pt,line join=round] ( 49.58,194.79) --
	(211.31,194.79);

\path[draw=drawColor,line width= 0.3pt,line join=round] ( 93.69, 36.18) --
	( 93.69,211.31);

\path[draw=drawColor,line width= 0.3pt,line join=round] (167.20, 36.18) --
	(167.20,211.31);

\path[draw=drawColor,line width= 0.3pt,line join=round] (203.96, 36.18) --
	(203.96,211.31);

\path[draw=drawColor,line width= 0.6pt,line join=round] ( 49.58, 69.88) --
	(211.31, 69.88);

\path[draw=drawColor,line width= 0.6pt,line join=round] ( 49.58,105.57) --
	(211.31,105.57);

\path[draw=drawColor,line width= 0.6pt,line join=round] ( 49.58,141.26) --
	(211.31,141.26);

\path[draw=drawColor,line width= 0.6pt,line join=round] ( 49.58,176.95) --
	(211.31,176.95);

\path[draw=drawColor,line width= 0.6pt,line join=round] ( 56.93, 36.18) --
	( 56.93,211.31);

\path[draw=drawColor,line width= 0.6pt,line join=round] (130.44, 36.18) --
	(130.44,211.31);
\definecolor{drawColor}{RGB}{199,124,255}

\path[draw=drawColor,line width= 0.4pt,line join=round,line cap=round] ( 54.15, 44.14) -- ( 59.70, 44.14);

\path[draw=drawColor,line width= 0.4pt,line join=round,line cap=round] ( 56.93, 41.37) -- ( 56.93, 46.92);

\path[draw=drawColor,line width= 0.4pt,line join=round,line cap=round] (127.67,107.33) -- (133.22,107.33);

\path[draw=drawColor,line width= 0.4pt,line join=round,line cap=round] (130.44,104.55) -- (130.44,110.10);

\path[draw=drawColor,line width= 0.4pt,line join=round,line cap=round] (162.74,132.83) -- (168.29,132.83);

\path[draw=drawColor,line width= 0.4pt,line join=round,line cap=round] (165.52,130.05) -- (165.52,135.60);

\path[draw=drawColor,line width= 0.4pt,line join=round,line cap=round] (201.18,175.73) -- (206.73,175.73);

\path[draw=drawColor,line width= 0.4pt,line join=round,line cap=round] (203.96,172.95) -- (203.96,178.50);
\definecolor{fillColor}{RGB}{0,191,196}

\path[fill=fillColor] ( 54.97, 65.76) --
	( 58.89, 65.76) --
	( 58.89, 69.69) --
	( 54.97, 69.69) --
	cycle;

\path[fill=fillColor] (128.48,128.29) --
	(132.41,128.29) --
	(132.41,132.22) --
	(128.48,132.22) --
	cycle;

\path[fill=fillColor] (163.56,144.90) --
	(167.48,144.90) --
	(167.48,148.82) --
	(163.56,148.82) --
	cycle;

\path[fill=fillColor] (202.00,180.67) --
	(205.92,180.67) --
	(205.92,184.59) --
	(202.00,184.59) --
	cycle;
\definecolor{fillColor}{RGB}{248,118,109}

\path[fill=fillColor] ( 56.93, 46.41) circle (  1.96);

\path[fill=fillColor] (130.44,120.95) circle (  1.96);

\path[fill=fillColor] (165.52,155.72) circle (  1.96);

\path[fill=fillColor] (203.96,203.35) circle (  1.96);
\definecolor{fillColor}{RGB}{124,174,0}

\path[fill=fillColor] ( 56.93, 49.46) --
	( 59.57, 44.89) --
	( 54.29, 44.89) --
	cycle;

\path[fill=fillColor] (130.44,122.31) --
	(133.09,117.73) --
	(127.80,117.73) --
	cycle;

\path[fill=fillColor] (165.52,149.37) --
	(168.16,144.79) --
	(162.88,144.79) --
	cycle;

\path[fill=fillColor] (203.96,196.34) --
	(206.60,191.76) --
	(201.32,191.76) --
	cycle;
\definecolor{drawColor}{RGB}{248,118,109}

\path[draw=drawColor,line width= 0.6pt,line join=round] ( 56.93, 46.41) --
	(130.44,120.95) --
	(165.52,155.72) --
	(203.96,203.35);
\definecolor{drawColor}{RGB}{124,174,0}

\path[draw=drawColor,line width= 0.6pt,line join=round] ( 56.93, 46.41) --
	(130.44,119.26) --
	(165.52,146.32) --
	(203.96,193.29);
\definecolor{drawColor}{RGB}{0,191,196}

\path[draw=drawColor,line width= 0.6pt,line join=round] ( 56.93, 67.72) --
	(130.44,130.26) --
	(165.52,146.86) --
	(203.96,182.63);
\definecolor{drawColor}{RGB}{199,124,255}

\path[draw=drawColor,line width= 0.6pt,line join=round] ( 56.93, 44.14) --
	(130.44,107.33) --
	(165.52,132.83) --
	(203.96,175.73);
\end{scope}
\begin{scope}
\path[clip] (  0.00,  0.00) rectangle (216.81,216.81);
\definecolor{drawColor}{gray}{0.30}

\node[text=drawColor,anchor=base east,inner sep=0pt, outer sep=0pt, scale=  1.20] at ( 44.63, 65.75) {1};

\node[text=drawColor,anchor=base east,inner sep=0pt, outer sep=0pt, scale=  1.20] at ( 44.63,101.44) {10};

\node[text=drawColor,anchor=base east,inner sep=0pt, outer sep=0pt, scale=  1.20] at ( 44.63,137.13) {100};

\node[text=drawColor,anchor=base east,inner sep=0pt, outer sep=0pt, scale=  1.20] at ( 44.63,172.82) {1000};
\end{scope}
\begin{scope}
\path[clip] (  0.00,  0.00) rectangle (216.81,216.81);
\definecolor{drawColor}{gray}{0.20}

\path[draw=drawColor,line width= 0.6pt,line join=round] ( 46.83, 69.88) --
	( 49.58, 69.88);

\path[draw=drawColor,line width= 0.6pt,line join=round] ( 46.83,105.57) --
	( 49.58,105.57);

\path[draw=drawColor,line width= 0.6pt,line join=round] ( 46.83,141.26) --
	( 49.58,141.26);

\path[draw=drawColor,line width= 0.6pt,line join=round] ( 46.83,176.95) --
	( 49.58,176.95);
\end{scope}
\begin{scope}
\path[clip] (  0.00,  0.00) rectangle (216.81,216.81);
\definecolor{drawColor}{gray}{0.20}

\path[draw=drawColor,line width= 0.6pt,line join=round] ( 56.93, 33.43) --
	( 56.93, 36.18);

\path[draw=drawColor,line width= 0.6pt,line join=round] (130.44, 33.43) --
	(130.44, 36.18);
\end{scope}
\begin{scope}
\path[clip] (  0.00,  0.00) rectangle (216.81,216.81);
\definecolor{drawColor}{gray}{0.30}

\node[text=drawColor,anchor=base,inner sep=0pt, outer sep=0pt, scale=  1] at ( 56.93, 22.97) {100};

\node[text=drawColor,anchor=base,inner sep=0pt, outer sep=0pt, scale=  1] at (130.44, 22.97) {1000};
\end{scope}
\begin{scope}
\path[clip] (  0.00,  0.00) rectangle (216.81,216.81);
\definecolor{drawColor}{RGB}{0,0,0}

\node[text=drawColor,anchor=base,inner sep=0pt, outer sep=0pt, scale=  1] at (130.44,  8.22) {\bfseries sample size};
\end{scope}
\begin{scope}
\path[clip] (  0.00,  0.00) rectangle (216.81,216.81);
\definecolor{drawColor}{RGB}{0,0,0}

\node[text=drawColor,rotate= 90.00,anchor=base,inner sep=0pt, outer sep=0pt, scale=  1] at ( 15.16,123.75) {\bfseries time[s]};
\end{scope}
\begin{scope}
\path[clip] (  0.00,  0.00) rectangle (216.81,216.81);
\definecolor{fillColor}{RGB}{255,255,255}

\path[fill=fillColor] ( 49.58,136.99) rectangle (130,211.31);
\end{scope}
\begin{scope}
\path[clip] (  0.00,  0.00) rectangle (216.81,216.81);
\definecolor{fillColor}{gray}{0.95}

\path[fill=fillColor] ( 55.08,185.86) rectangle ( 69.53,200.31);
\end{scope}
\begin{scope}
\path[clip] (  0.00,  0.00) rectangle (216.81,216.81);
\definecolor{fillColor}{RGB}{248,118,109}

\path[fill=fillColor] ( 62.30,193.08) circle (  1.96);
\end{scope}
\begin{scope}
\path[clip] (  0.00,  0.00) rectangle (216.81,216.81);
\definecolor{drawColor}{RGB}{248,118,109}

\path[draw=drawColor,line width= 0.6pt,line join=round] ( 56.52,193.08) -- ( 68.09,193.08);
\end{scope}
\begin{scope}
\path[clip] (  0.00,  0.00) rectangle (216.81,216.81);
\definecolor{fillColor}{gray}{0.95}

\path[fill=fillColor] ( 55.08,171.40) rectangle ( 69.53,185.86);
\end{scope}
\begin{scope}
\path[clip] (  0.00,  0.00) rectangle (216.81,216.81);
\definecolor{fillColor}{RGB}{124,174,0}

\path[fill=fillColor] ( 62.30,181.68) --
	( 64.95,177.10) --
	( 59.66,177.10) --
	cycle;
\end{scope}
\begin{scope}
\path[clip] (  0.00,  0.00) rectangle (216.81,216.81);
\definecolor{drawColor}{RGB}{124,174,0}

\path[draw=drawColor,line width= 0.6pt,line join=round] ( 56.52,178.63) -- ( 68.09,178.63);
\end{scope}
\begin{scope}
\path[clip] (  0.00,  0.00) rectangle (216.81,216.81);
\definecolor{fillColor}{gray}{0.95}

\path[fill=fillColor] ( 55.08,156.95) rectangle ( 69.53,171.40);
\end{scope}
\begin{scope}
\path[clip] (  0.00,  0.00) rectangle (216.81,216.81);
\definecolor{fillColor}{RGB}{0,191,196}

\path[fill=fillColor] ( 60.34,162.21) --
	( 64.27,162.21) --
	( 64.27,166.14) --
	( 60.34,166.14) --
	cycle;
\end{scope}
\begin{scope}
\path[clip] (  0.00,  0.00) rectangle (216.81,216.81);
\definecolor{drawColor}{RGB}{0,191,196}

\path[draw=drawColor,line width= 0.6pt,line join=round] ( 56.52,164.18) -- ( 68.09,164.18);
\end{scope}
\begin{scope}
\path[clip] (  0.00,  0.00) rectangle (216.81,216.81);
\definecolor{fillColor}{gray}{0.95}

\path[fill=fillColor] ( 55.08,142.49) rectangle ( 69.53,156.95);
\end{scope}
\begin{scope}
\path[clip] (  0.00,  0.00) rectangle (216.81,216.81);
\definecolor{drawColor}{RGB}{199,124,255}

\path[draw=drawColor,line width= 0.4pt,line join=round,line cap=round] ( 59.53,149.72) -- ( 65.08,149.72);

\path[draw=drawColor,line width= 0.4pt,line join=round,line cap=round] ( 62.30,146.95) -- ( 62.30,152.50);
\end{scope}
\begin{scope}
\path[clip] (  0.00,  0.00) rectangle (216.81,216.81);
\definecolor{drawColor}{RGB}{199,124,255}

\path[draw=drawColor,line width= 0.6pt,line join=round] ( 56.52,149.72) -- ( 68.09,149.72);
\end{scope}
\begin{scope}
\path[clip] (  0.00,  0.00) rectangle (216.81,216.81);
\definecolor{drawColor}{RGB}{0,0,0}

\node[text=drawColor,anchor=base west,inner sep=0pt, outer sep=0pt, scale=  .8] at ( 75.03,188.95) {AMA};
\end{scope}
\begin{scope}
\path[clip] (  0.00,  0.00) rectangle (216.81,216.81);
\definecolor{drawColor}{RGB}{0,0,0}

\node[text=drawColor,anchor=base west,inner sep=0pt, outer sep=0pt, scale=  0.8] at ( 75.03,174.50) {AMA-FISTA};
\end{scope}
\begin{scope}
\path[clip] (  0.00,  0.00) rectangle (216.81,216.81);
\definecolor{drawColor}{RGB}{0,0,0}

\node[text=drawColor,anchor=base west,inner sep=0pt, outer sep=0pt, scale=  .8] at ( 75.03,160.04) {genADMM};
\end{scope}
\begin{scope}
\path[clip] (  0.00,  0.00) rectangle (216.81,216.81);
\definecolor{drawColor}{RGB}{0,0,0}

\node[text=drawColor,anchor=base west,inner sep=0pt, outer sep=0pt, scale=  .8] at ( 75.03,145.59) {proposed};
\end{scope}
\end{tikzpicture}

%% file: fig_4.tex
\begin{tikzpicture}[x=0.8pt,y=0.8pt]
\definecolor{fillColor}{RGB}{255,255,255}
\path[use as bounding box,fill=fillColor,fill opacity=0.00] (0,0) rectangle (216.81,216.81);
\begin{scope}
\path[clip] (  0.00,  0.00) rectangle (216.81,216.81);
\definecolor{drawColor}{RGB}{255,255,255}
\definecolor{fillColor}{RGB}{255,255,255}

\path[draw=drawColor,line width= 0.6pt,line join=round,line cap=round,fill=fillColor] (  0.00,  0.00) rectangle (216.81,216.81);
\end{scope}
\begin{scope}
\path[clip] ( 43.58, 36.18) rectangle (211.31,211.31);
\definecolor{fillColor}{gray}{0.92}

\path[fill=fillColor] ( 43.58, 36.18) rectangle (211.31,211.31);
\definecolor{drawColor}{RGB}{255,255,255}

\path[draw=drawColor,line width= 0.3pt,line join=round] ( 43.58, 51.45) --
	(211.31, 51.45);

\path[draw=drawColor,line width= 0.3pt,line join=round] ( 43.58, 75.02) --
	(211.31, 75.02);

\path[draw=drawColor,line width= 0.3pt,line join=round] ( 43.58,122.18) --
	(211.31,122.18);

\path[draw=drawColor,line width= 0.3pt,line join=round] ( 43.58,169.33) --
	(211.31,169.33);

\path[draw=drawColor,line width= 0.3pt,line join=round] ( 66.37, 36.18) --
	( 66.37,211.31);

\path[draw=drawColor,line width= 0.3pt,line join=round] (142.61, 36.18) --
	(142.61,211.31);

\path[draw=drawColor,line width= 0.6pt,line join=round] ( 43.58, 98.60) --
	(211.31, 98.60);

\path[draw=drawColor,line width= 0.6pt,line join=round] ( 43.58,145.76) --
	(211.31,145.76);

\path[draw=drawColor,line width= 0.6pt,line join=round] ( 43.58,192.91) --
	(211.31,192.91);

\path[draw=drawColor,line width= 0.6pt,line join=round] (104.49, 36.18) --
	(104.49,211.31);

\path[draw=drawColor,line width= 0.6pt,line join=round] (180.73, 36.18) --
	(180.73,211.31);
\definecolor{drawColor}{RGB}{199,124,255}

\path[draw=drawColor,line width= 0.4pt,line join=round,line cap=round] ( 48.43, 44.14) -- ( 53.98, 44.14);

\path[draw=drawColor,line width= 0.4pt,line join=round,line cap=round] ( 51.20, 41.37) -- ( 51.20, 46.92);

\path[draw=drawColor,line width= 0.4pt,line join=round,line cap=round] (101.72, 83.14) -- (107.27, 83.14);

\path[draw=drawColor,line width= 0.4pt,line join=round,line cap=round] (104.49, 80.36) -- (104.49, 85.91);

\path[draw=drawColor,line width= 0.4pt,line join=round,line cap=round] (138.09,107.18) -- (143.64,107.18);

\path[draw=drawColor,line width= 0.4pt,line join=round,line cap=round] (140.87,104.40) -- (140.87,109.95);

\path[draw=drawColor,line width= 0.4pt,line join=round,line cap=round] (177.96,143.76) -- (183.51,143.76);

\path[draw=drawColor,line width= 0.4pt,line join=round,line cap=round] (180.73,140.98) -- (180.73,146.53);

\path[draw=drawColor,line width= 0.4pt,line join=round,line cap=round] (200.91,160.75) -- (206.46,160.75);

\path[draw=drawColor,line width= 0.4pt,line join=round,line cap=round] (203.69,157.98) -- (203.69,163.53);
\definecolor{fillColor}{RGB}{0,191,196}

\path[fill=fillColor] ( 49.24, 47.33) --
	( 53.16, 47.33) --
	( 53.16, 51.25) --
	( 49.24, 51.25) --
	cycle;

\path[fill=fillColor] (102.53, 92.07) --
	(106.46, 92.07) --
	(106.46, 95.99) --
	(102.53, 95.99) --
	cycle;

\path[fill=fillColor] (138.91,144.68) --
	(142.83,144.68) --
	(142.83,148.60) --
	(138.91,148.60) --
	cycle;

\path[fill=fillColor] (178.77,185.88) --
	(182.70,185.88) --
	(182.70,189.81) --
	(178.77,189.81) --
	cycle;

\path[fill=fillColor] (201.72,201.39) --
	(205.65,201.39) --
	(205.65,205.31) --
	(201.72,205.31) --
	cycle;
\definecolor{fillColor}{RGB}{248,118,109}

\path[fill=fillColor] ( 51.20, 65.64) circle (  1.96);

\path[fill=fillColor] (104.49,104.44) circle (  1.96);

\path[fill=fillColor] (140.87,128.23) circle (  1.96);

\path[fill=fillColor] (180.73,141.34) circle (  1.96);

\path[fill=fillColor] (203.69,155.89) circle (  1.96);
\definecolor{fillColor}{RGB}{124,174,0}

\path[fill=fillColor] ( 51.20, 67.64) --
	( 53.85, 63.06) --
	( 48.56, 63.06) --
	cycle;

\path[fill=fillColor] (104.49, 95.76) --
	(107.14, 91.18) --
	(101.85, 91.18) --
	cycle;

\path[fill=fillColor] (140.87,131.76) --
	(143.51,127.18) --
	(138.23,127.18) --
	cycle;

\path[fill=fillColor] (180.73,144.36) --
	(183.38,139.79) --
	(178.09,139.79) --
	cycle;

\path[fill=fillColor] (203.69,159.04) --
	(206.33,154.46) --
	(201.04,154.46) --
	cycle;
\definecolor{drawColor}{RGB}{248,118,109}

\path[draw=drawColor,line width= 0.6pt,line join=round] ( 51.20, 65.64) --
	(104.49,104.44) --
	(140.87,128.23) --
	(180.73,141.34) --
	(203.69,155.89);
\definecolor{drawColor}{RGB}{124,174,0}

\path[draw=drawColor,line width= 0.6pt,line join=round] ( 51.20, 64.59) --
	(104.49, 92.71) --
	(140.87,128.71) --
	(180.73,141.31) --
	(203.69,155.99);
\definecolor{drawColor}{RGB}{0,191,196}

\path[draw=drawColor,line width= 0.6pt,line join=round] ( 51.20, 49.29) --
	(104.49, 94.03) --
	(140.87,146.64) --
	(180.73,187.84) --
	(203.69,203.35);
\definecolor{drawColor}{RGB}{199,124,255}

\path[draw=drawColor,line width= 0.6pt,line join=round] ( 51.20, 44.14) --
	(104.49, 83.14) --
	(140.87,107.18) --
	(180.73,143.76) --
	(203.69,160.75);
\end{scope}
\begin{scope}
\path[clip] (  0.00,  0.00) rectangle (216.81,216.81);
\definecolor{drawColor}{gray}{0.30}

\node[text=drawColor,anchor=base east,inner sep=0pt, outer sep=0pt, scale=  1.20] at ( 38.63, 94.47) {1};

\node[text=drawColor,anchor=base east,inner sep=0pt, outer sep=0pt, scale=  1.20] at ( 38.63,141.62) {10};

\node[text=drawColor,anchor=base east,inner sep=0pt, outer sep=0pt, scale=  1.20] at ( 38.63,188.78) {100};
\end{scope}
\begin{scope}
\path[clip] (  0.00,  0.00) rectangle (216.81,216.81);
\definecolor{drawColor}{gray}{0.20}

\path[draw=drawColor,line width= 0.6pt,line join=round] ( 40.83, 98.60) --
	( 43.58, 98.60);

\path[draw=drawColor,line width= 0.6pt,line join=round] ( 40.83,145.76) --
	( 43.58,145.76);

\path[draw=drawColor,line width= 0.6pt,line join=round] ( 40.83,192.91) --
	( 43.58,192.91);
\end{scope}
\begin{scope}
\path[clip] (  0.00,  0.00) rectangle (216.81,216.81);
\definecolor{drawColor}{gray}{0.20}

\path[draw=drawColor,line width= 0.6pt,line join=round] (104.49, 33.43) --
	(104.49, 36.18);

\path[draw=drawColor,line width= 0.6pt,line join=round] (180.73, 33.43) --
	(180.73, 36.18);
\end{scope}
\begin{scope}
\path[clip] (  0.00,  0.00) rectangle (216.81,216.81);
\definecolor{drawColor}{gray}{0.30}

\node[text=drawColor,anchor=base,inner sep=0pt, outer sep=0pt, scale=  1] at (104.49, 22.97) {100};

\node[text=drawColor,anchor=base,inner sep=0pt, outer sep=0pt, scale=  1] at (180.73, 22.97) {1000};
\end{scope}
\begin{scope}
\path[clip] (  0.00,  0.00) rectangle (216.81,216.81);
\definecolor{drawColor}{RGB}{0,0,0}

\node[text=drawColor,anchor=base,inner sep=0pt, outer sep=0pt, scale=  1] at (127.44,  8.22) {\bfseries feature dimension};
\end{scope}
\begin{scope}
\path[clip] (  0.00,  0.00) rectangle (216.81,216.81);
\definecolor{drawColor}{RGB}{0,0,0}

\node[text=drawColor,rotate= 90.00,anchor=base,inner sep=0pt, outer sep=0pt, scale=  1] at ( 15.16,123.75) {\bfseries time[s]};
\end{scope}
\begin{scope}
\path[clip] (  0.00,  0.00) rectangle (216.81,216.81);
\definecolor{fillColor}{RGB}{255,255,255}

\path[fill=fillColor] ( 43.58,136.99) rectangle (130,211.31);
\end{scope}
\begin{scope}
\path[clip] (  0.00,  0.00) rectangle (216.81,216.81);
\definecolor{fillColor}{gray}{0.95}

\path[fill=fillColor] ( 49.08,185.86) rectangle ( 63.53,200.31);
\end{scope}
\begin{scope}
\path[clip] (  0.00,  0.00) rectangle (216.81,216.81);
\definecolor{fillColor}{RGB}{248,118,109}

\path[fill=fillColor] ( 56.31,193.08) circle (  1.96);
\end{scope}
\begin{scope}
\path[clip] (  0.00,  0.00) rectangle (216.81,216.81);
\definecolor{drawColor}{RGB}{248,118,109}

\path[draw=drawColor,line width= 0.6pt,line join=round] ( 50.52,193.08) -- ( 62.09,193.08);
\end{scope}
\begin{scope}
\path[clip] (  0.00,  0.00) rectangle (216.81,216.81);
\definecolor{fillColor}{gray}{0.95}

\path[fill=fillColor] ( 49.08,171.40) rectangle ( 63.53,185.86);
\end{scope}
\begin{scope}
\path[clip] (  0.00,  0.00) rectangle (216.81,216.81);
\definecolor{fillColor}{RGB}{124,174,0}

\path[fill=fillColor] ( 56.31,181.68) --
	( 58.95,177.10) --
	( 53.66,177.10) --
	cycle;
\end{scope}
\begin{scope}
\path[clip] (  0.00,  0.00) rectangle (216.81,216.81);
\definecolor{drawColor}{RGB}{124,174,0}

\path[draw=drawColor,line width= 0.6pt,line join=round] ( 50.52,178.63) -- ( 62.09,178.63);
\end{scope}
\begin{scope}
\path[clip] (  0.00,  0.00) rectangle (216.81,216.81);
\definecolor{fillColor}{gray}{0.95}

\path[fill=fillColor] ( 49.08,156.95) rectangle ( 63.53,171.40);
\end{scope}
\begin{scope}
\path[clip] (  0.00,  0.00) rectangle (216.81,216.81);
\definecolor{fillColor}{RGB}{0,191,196}

\path[fill=fillColor] ( 54.34,162.21) --
	( 58.27,162.21) --
	( 58.27,166.14) --
	( 54.34,166.14) --
	cycle;
\end{scope}
\begin{scope}
\path[clip] (  0.00,  0.00) rectangle (216.81,216.81);
\definecolor{drawColor}{RGB}{0,191,196}

\path[draw=drawColor,line width= 0.6pt,line join=round] ( 50.52,164.18) -- ( 62.09,164.18);
\end{scope}
\begin{scope}
\path[clip] (  0.00,  0.00) rectangle (216.81,216.81);
\definecolor{fillColor}{gray}{0.95}

\path[fill=fillColor] ( 49.08,142.49) rectangle ( 63.53,156.95);
\end{scope}
\begin{scope}
\path[clip] (  0.00,  0.00) rectangle (216.81,216.81);
\definecolor{drawColor}{RGB}{199,124,255}

\path[draw=drawColor,line width= 0.4pt,line join=round,line cap=round] ( 53.53,149.72) -- ( 59.08,149.72);

\path[draw=drawColor,line width= 0.4pt,line join=round,line cap=round] ( 56.31,146.95) -- ( 56.31,152.50);
\end{scope}
\begin{scope}
\path[clip] (  0.00,  0.00) rectangle (216.81,216.81);
\definecolor{drawColor}{RGB}{199,124,255}

\path[draw=drawColor,line width= 0.6pt,line join=round] ( 50.52,149.72) -- ( 62.09,149.72);
\end{scope}
\begin{scope}
\path[clip] (  0.00,  0.00) rectangle (216.81,216.81);
\definecolor{drawColor}{RGB}{0,0,0}

\node[text=drawColor,anchor=base west,inner sep=0pt, outer sep=0pt, scale=  .8] at ( 69.03,188.95) {AMA};
\end{scope}
\begin{scope}
\path[clip] (  0.00,  0.00) rectangle (216.81,216.81);
\definecolor{drawColor}{RGB}{0,0,0}

\node[text=drawColor,anchor=base west,inner sep=0pt, outer sep=0pt, scale=  .8] at ( 69.03,174.50) {AMA-FISTA};
\end{scope}
\begin{scope}
\path[clip] (  0.00,  0.00) rectangle (216.81,216.81);
\definecolor{drawColor}{RGB}{0,0,0}

\node[text=drawColor,anchor=base west,inner sep=0pt, outer sep=0pt, scale=  .8] at ( 69.03,160.04) {genADMM};
\end{scope}
\begin{scope}
\path[clip] (  0.00,  0.00) rectangle (216.81,216.81);
\definecolor{drawColor}{RGB}{0,0,0}

\node[text=drawColor,anchor=base west,inner sep=0pt, outer sep=0pt, scale=  .8] at ( 69.03,145.59) {proposed};
\end{scope}
\end{tikzpicture}

%% file: fig_7.tex
\begin{tikzpicture}[x=0.8pt,y=0.8pt]
\definecolor{fillColor}{RGB}{255,255,255}
\path[use as bounding box,fill=fillColor,fill opacity=0.00] (0,0) rectangle (216.81,216.81);
\begin{scope}
\path[clip] (  0.00,  0.00) rectangle (216.81,216.81);
\definecolor{drawColor}{RGB}{255,255,255}
\definecolor{fillColor}{RGB}{255,255,255}

\path[draw=drawColor,line width= 0.6pt,line join=round,line cap=round,fill=fillColor] (  0.00,  0.00) rectangle (216.81,216.81);
\end{scope}
\begin{scope}
\path[clip] ( 37.58, 36.18) rectangle (211.31,211.31);
\definecolor{fillColor}{gray}{0.92}

\path[fill=fillColor] ( 37.58, 36.18) rectangle (211.31,211.31);
\definecolor{drawColor}{RGB}{255,255,255}

\path[draw=drawColor,line width= 0.3pt,line join=round] ( 37.58, 62.87) --
	(211.31, 62.87);

\path[draw=drawColor,line width= 0.3pt,line join=round] ( 37.58,104.81) --
	(211.31,104.81);

\path[draw=drawColor,line width= 0.3pt,line join=round] ( 37.58,146.74) --
	(211.31,146.74);

\path[draw=drawColor,line width= 0.3pt,line join=round] ( 37.58,188.67) --
	(211.31,188.67);

\path[draw=drawColor,line width= 0.3pt,line join=round] ( 78.53, 36.18) --
	( 78.53,211.31);

\path[draw=drawColor,line width= 0.3pt,line join=round] (144.64, 36.18) --
	(144.64,211.31);

\path[draw=drawColor,line width= 0.3pt,line join=round] (210.75, 36.18) --
	(210.75,211.31);

\path[draw=drawColor,line width= 0.6pt,line join=round] ( 37.58, 41.90) --
	(211.31, 41.90);

\path[draw=drawColor,line width= 0.6pt,line join=round] ( 37.58, 83.84) --
	(211.31, 83.84);

\path[draw=drawColor,line width= 0.6pt,line join=round] ( 37.58,125.77) --
	(211.31,125.77);

\path[draw=drawColor,line width= 0.6pt,line join=round] ( 37.58,167.71) --
	(211.31,167.71);

\path[draw=drawColor,line width= 0.6pt,line join=round] ( 37.58,209.64) --
	(211.31,209.64);

\path[draw=drawColor,line width= 0.6pt,line join=round] ( 45.48, 36.18) --
	( 45.48,211.31);

\path[draw=drawColor,line width= 0.6pt,line join=round] (111.58, 36.18) --
	(111.58,211.31);

\path[draw=drawColor,line width= 0.6pt,line join=round] (177.69, 36.18) --
	(177.69,211.31);
\definecolor{drawColor}{RGB}{0,191,196}

\path[draw=drawColor,line width= 0.6pt,line join=round] ( 45.48, 44.14) --
	( 53.37, 45.82) --
	( 61.27, 49.31) --
	( 69.17, 56.72) --
	( 77.06, 78.25) --
	( 84.96, 82.72) --
	( 92.86, 86.35) --
	(100.75, 84.68) --
	(108.65, 86.77) --
	(116.55, 88.45) --
	(124.45, 87.47) --
	(132.34, 86.91) --
	(140.24, 88.59) --
	(148.14, 84.68) --
	(156.03, 86.91) --
	(163.93, 85.10) --
	(171.83, 86.91) --
	(179.72, 87.33) --
	(187.62, 87.33) --
	(195.52, 88.45) --
	(203.41, 87.19);

\path[draw=drawColor,line width= 0.4pt,line join=round,line cap=round] ( 45.48, 44.14) circle (  1.96);

\path[draw=drawColor,line width= 0.4pt,line join=round,line cap=round] ( 53.37, 45.82) circle (  1.96);

\path[draw=drawColor,line width= 0.4pt,line join=round,line cap=round] ( 61.27, 49.31) circle (  1.96);

\path[draw=drawColor,line width= 0.4pt,line join=round,line cap=round] ( 69.17, 56.72) circle (  1.96);

\path[draw=drawColor,line width= 0.4pt,line join=round,line cap=round] ( 77.06, 78.25) circle (  1.96);

\path[draw=drawColor,line width= 0.4pt,line join=round,line cap=round] ( 84.96, 82.72) circle (  1.96);

\path[draw=drawColor,line width= 0.4pt,line join=round,line cap=round] ( 92.86, 86.35) circle (  1.96);

\path[draw=drawColor,line width= 0.4pt,line join=round,line cap=round] (100.75, 84.68) circle (  1.96);

\path[draw=drawColor,line width= 0.4pt,line join=round,line cap=round] (108.65, 86.77) circle (  1.96);

\path[draw=drawColor,line width= 0.4pt,line join=round,line cap=round] (116.55, 88.45) circle (  1.96);

\path[draw=drawColor,line width= 0.4pt,line join=round,line cap=round] (124.45, 87.47) circle (  1.96);

\path[draw=drawColor,line width= 0.4pt,line join=round,line cap=round] (132.34, 86.91) circle (  1.96);

\path[draw=drawColor,line width= 0.4pt,line join=round,line cap=round] (140.24, 88.59) circle (  1.96);

\path[draw=drawColor,line width= 0.4pt,line join=round,line cap=round] (148.14, 84.68) circle (  1.96);

\path[draw=drawColor,line width= 0.4pt,line join=round,line cap=round] (156.03, 86.91) circle (  1.96);

\path[draw=drawColor,line width= 0.4pt,line join=round,line cap=round] (163.93, 85.10) circle (  1.96);

\path[draw=drawColor,line width= 0.4pt,line join=round,line cap=round] (171.83, 86.91) circle (  1.96);

\path[draw=drawColor,line width= 0.4pt,line join=round,line cap=round] (179.72, 87.33) circle (  1.96);

\path[draw=drawColor,line width= 0.4pt,line join=round,line cap=round] (187.62, 87.33) circle (  1.96);

\path[draw=drawColor,line width= 0.4pt,line join=round,line cap=round] (195.52, 88.45) circle (  1.96);

\path[draw=drawColor,line width= 0.4pt,line join=round,line cap=round] (203.41, 87.19) circle (  1.96);
\definecolor{drawColor}{RGB}{248,118,109}

\path[draw=drawColor,line width= 0.6pt,line join=round] ( 45.48, 49.31) --
	( 53.37, 54.07) --
	( 61.27, 62.87) --
	( 69.17, 86.49) --
	( 77.06,109.42) --
	( 84.96,149.53) --
	( 92.86,179.87) --
	(100.75,203.35) --
	(108.65,203.07) --
	(116.55,202.37) --
	(124.45,203.21) --
	(132.34,201.11) --
	(140.24,202.79) --
	(148.14,203.35) --
	(156.03,202.65) --
	(163.93,202.51) --
	(171.83,203.21) --
	(179.72,202.79) --
	(187.62,202.65) --
	(195.52,202.93) --
	(203.41,202.23);

\path[draw=drawColor,line width= 0.4pt,line join=round,line cap=round] ( 45.48, 49.31) circle (  1.96);

\path[draw=drawColor,line width= 0.4pt,line join=round,line cap=round] ( 53.37, 54.07) circle (  1.96);

\path[draw=drawColor,line width= 0.4pt,line join=round,line cap=round] ( 61.27, 62.87) circle (  1.96);

\path[draw=drawColor,line width= 0.4pt,line join=round,line cap=round] ( 69.17, 86.49) circle (  1.96);

\path[draw=drawColor,line width= 0.4pt,line join=round,line cap=round] ( 77.06,109.42) circle (  1.96);

\path[draw=drawColor,line width= 0.4pt,line join=round,line cap=round] ( 84.96,149.53) circle (  1.96);

\path[draw=drawColor,line width= 0.4pt,line join=round,line cap=round] ( 92.86,179.87) circle (  1.96);

\path[draw=drawColor,line width= 0.4pt,line join=round,line cap=round] (100.75,203.35) circle (  1.96);

\path[draw=drawColor,line width= 0.4pt,line join=round,line cap=round] (108.65,203.07) circle (  1.96);

\path[draw=drawColor,line width= 0.4pt,line join=round,line cap=round] (116.55,202.37) circle (  1.96);

\path[draw=drawColor,line width= 0.4pt,line join=round,line cap=round] (124.45,203.21) circle (  1.96);

\path[draw=drawColor,line width= 0.4pt,line join=round,line cap=round] (132.34,201.11) circle (  1.96);

\path[draw=drawColor,line width= 0.4pt,line join=round,line cap=round] (140.24,202.79) circle (  1.96);

\path[draw=drawColor,line width= 0.4pt,line join=round,line cap=round] (148.14,203.35) circle (  1.96);

\path[draw=drawColor,line width= 0.4pt,line join=round,line cap=round] (156.03,202.65) circle (  1.96);

\path[draw=drawColor,line width= 0.4pt,line join=round,line cap=round] (163.93,202.51) circle (  1.96);

\path[draw=drawColor,line width= 0.4pt,line join=round,line cap=round] (171.83,203.21) circle (  1.96);

\path[draw=drawColor,line width= 0.4pt,line join=round,line cap=round] (179.72,202.79) circle (  1.96);

\path[draw=drawColor,line width= 0.4pt,line join=round,line cap=round] (187.62,202.65) circle (  1.96);

\path[draw=drawColor,line width= 0.4pt,line join=round,line cap=round] (195.52,202.93) circle (  1.96);

\path[draw=drawColor,line width= 0.4pt,line join=round,line cap=round] (203.41,202.23) circle (  1.96);
\end{scope}
\begin{scope}
\path[clip] (  0.00,  0.00) rectangle (216.81,216.81);
\definecolor{drawColor}{gray}{0.30}

\node[text=drawColor,anchor=base east,inner sep=0pt, outer sep=0pt, scale=  1.20] at ( 32.63, 37.77) {0};

\node[text=drawColor,anchor=base east,inner sep=0pt, outer sep=0pt, scale=  1.20] at ( 32.63, 79.71) {3};

\node[text=drawColor,anchor=base east,inner sep=0pt, outer sep=0pt, scale=  1.20] at ( 32.63,121.64) {6};

\node[text=drawColor,anchor=base east,inner sep=0pt, outer sep=0pt, scale=  1.20] at ( 32.63,163.57) {9};

\node[text=drawColor,anchor=base east,inner sep=0pt, outer sep=0pt, scale=  1.20] at ( 32.63,205.51) {12};
\end{scope}
\begin{scope}
\path[clip] (  0.00,  0.00) rectangle (216.81,216.81);
\definecolor{drawColor}{gray}{0.20}

\path[draw=drawColor,line width= 0.6pt,line join=round] ( 34.83, 41.90) --
	( 37.58, 41.90);

\path[draw=drawColor,line width= 0.6pt,line join=round] ( 34.83, 83.84) --
	( 37.58, 83.84);

\path[draw=drawColor,line width= 0.6pt,line join=round] ( 34.83,125.77) --
	( 37.58,125.77);

\path[draw=drawColor,line width= 0.6pt,line join=round] ( 34.83,167.71) --
	( 37.58,167.71);

\path[draw=drawColor,line width= 0.6pt,line join=round] ( 34.83,209.64) --
	( 37.58,209.64);
\end{scope}
\begin{scope}
\path[clip] (  0.00,  0.00) rectangle (216.81,216.81);
\definecolor{drawColor}{gray}{0.20}

\path[draw=drawColor,line width= 0.6pt,line join=round] ( 45.48, 33.43) --
	( 45.48, 36.18);

\path[draw=drawColor,line width= 0.6pt,line join=round] (111.58, 33.43) --
	(111.58, 36.18);

\path[draw=drawColor,line width= 0.6pt,line join=round] (177.69, 33.43) --
	(177.69, 36.18);
\end{scope}
\begin{scope}
\path[clip] (  0.00,  0.00) rectangle (216.81,216.81);
\definecolor{drawColor}{gray}{0.30}

\node[text=drawColor,anchor=base,inner sep=0pt, outer sep=0pt, scale=  1.20] at ( 45.48, 22.97) {1};

\node[text=drawColor,anchor=base,inner sep=0pt, outer sep=0pt, scale=  1.20] at (111.58, 22.97) {100};

\node[text=drawColor,anchor=base,inner sep=0pt, outer sep=0pt, scale=  1.20] at (177.69, 22.97) {10000};
\end{scope}
\begin{scope}
\path[clip] (  0.00,  0.00) rectangle (216.81,216.81);
\definecolor{drawColor}{RGB}{0,0,0}

\node[text=drawColor,anchor=base,inner sep=0pt, outer sep=0pt, scale=  1.40] at (124.45,  8.22) {\bfseries $\gamma$};
\end{scope}
\begin{scope}
\path[clip] (  0.00,  0.00) rectangle (216.81,216.81);
\definecolor{drawColor}{RGB}{0,0,0}

\node[text=drawColor,rotate= 90.00,anchor=base,inner sep=0pt, outer sep=0pt, scale=  1.40] at ( 15.16,123.75) {\bfseries time[s]};
\end{scope}
\begin{scope}
\path[clip] (  0.00,  0.00) rectangle (216.81,216.81);
\definecolor{fillColor}{RGB}{255,255,255}

\path[fill=fillColor] (131.77, 36.18) rectangle (211.31, 81.59);
\end{scope}
\begin{scope}
\path[clip] (  0.00,  0.00) rectangle (216.81,216.81);
\definecolor{fillColor}{gray}{0.95}

\path[fill=fillColor] (137.27, 56.13) rectangle (151.72, 70.59);
\end{scope}
\begin{scope}
\path[clip] (  0.00,  0.00) rectangle (216.81,216.81);
\definecolor{drawColor}{RGB}{248,118,109}

\path[draw=drawColor,line width= 0.6pt,line join=round] (138.71, 63.36) -- (150.28, 63.36);
\end{scope}
\begin{scope}
\path[clip] (  0.00,  0.00) rectangle (216.81,216.81);
\definecolor{drawColor}{RGB}{248,118,109}

\path[draw=drawColor,line width= 0.4pt,line join=round,line cap=round] (144.50, 63.36) circle (  1.96);
\end{scope}
\begin{scope}
\path[clip] (  0.00,  0.00) rectangle (216.81,216.81);
\definecolor{drawColor}{RGB}{248,118,109}

\path[draw=drawColor,line width= 0.6pt,line join=round] (138.71, 63.36) -- (150.28, 63.36);
\end{scope}
\begin{scope}
\path[clip] (  0.00,  0.00) rectangle (216.81,216.81);
\definecolor{drawColor}{RGB}{248,118,109}

\path[draw=drawColor,line width= 0.4pt,line join=round,line cap=round] (144.50, 63.36) circle (  1.96);
\end{scope}
\begin{scope}
\path[clip] (  0.00,  0.00) rectangle (216.81,216.81);
\definecolor{fillColor}{gray}{0.95}

\path[fill=fillColor] (137.27, 41.68) rectangle (151.72, 56.13);
\end{scope}
\begin{scope}
\path[clip] (  0.00,  0.00) rectangle (216.81,216.81);
\definecolor{drawColor}{RGB}{0,191,196}

\path[draw=drawColor,line width= 0.6pt,line join=round] (138.71, 48.91) -- (150.28, 48.91);
\end{scope}
\begin{scope}
\path[clip] (  0.00,  0.00) rectangle (216.81,216.81);
\definecolor{drawColor}{RGB}{0,191,196}

\path[draw=drawColor,line width= 0.4pt,line join=round,line cap=round] (144.50, 48.91) circle (  1.96);
\end{scope}
\begin{scope}
\path[clip] (  0.00,  0.00) rectangle (216.81,216.81);
\definecolor{drawColor}{RGB}{0,191,196}

\path[draw=drawColor,line width= 0.6pt,line join=round] (138.71, 48.91) -- (150.28, 48.91);
\end{scope}
\begin{scope}
\path[clip] (  0.00,  0.00) rectangle (216.81,216.81);
\definecolor{drawColor}{RGB}{0,191,196}

\path[draw=drawColor,line width= 0.4pt,line join=round,line cap=round] (144.50, 48.91) circle (  1.96);
\end{scope}
\begin{scope}
\path[clip] (  0.00,  0.00) rectangle (216.81,216.81);
\definecolor{drawColor}{RGB}{0,0,0}

\node[text=drawColor,anchor=base west,inner sep=0pt, outer sep=0pt, scale=  1.20] at (157.22, 59.23) {ADMM};
\end{scope}
\begin{scope}
\path[clip] (  0.00,  0.00) rectangle (216.81,216.81);
\definecolor{drawColor}{RGB}{0,0,0}

\node[text=drawColor,anchor=base west,inner sep=0pt, outer sep=0pt, scale=  1.20] at (157.22, 44.78) {Proposed};
\end{scope}
\end{tikzpicture}

%% file: fig_8.tex
\begin{tikzpicture}[x=0.8pt,y=0.8pt]
\definecolor{fillColor}{RGB}{255,255,255}
\path[use as bounding box,fill=fillColor,fill opacity=0.00] (0,0) rectangle (216.81,216.81);
\begin{scope}
\path[clip] (  0.00,  0.00) rectangle (216.81,216.81);
\definecolor{drawColor}{RGB}{255,255,255}
\definecolor{fillColor}{RGB}{255,255,255}

\path[draw=drawColor,line width= 0.6pt,line join=round,line cap=round,fill=fillColor] (  0.00,  0.00) rectangle (216.81,216.81);
\end{scope}
\begin{scope}
\path[clip] ( 37.58, 36.18) rectangle (211.31,211.31);
\definecolor{fillColor}{gray}{0.92}

\path[fill=fillColor] ( 37.58, 36.18) rectangle (211.31,211.31);
\definecolor{drawColor}{RGB}{255,255,255}

\path[draw=drawColor,line width= 0.3pt,line join=round] ( 37.58, 65.00) --
	(211.31, 65.00);

\path[draw=drawColor,line width= 0.3pt,line join=round] ( 37.58,113.95) --
	(211.31,113.95);

\path[draw=drawColor,line width= 0.3pt,line join=round] ( 37.58,162.91) --
	(211.31,162.91);

\path[draw=drawColor,line width= 0.3pt,line join=round] ( 45.48, 36.18) --
	( 45.48,211.31);

\path[draw=drawColor,line width= 0.3pt,line join=round] ( 88.74, 36.18) --
	( 88.74,211.31);

\path[draw=drawColor,line width= 0.3pt,line join=round] (132.01, 36.18) --
	(132.01,211.31);

\path[draw=drawColor,line width= 0.3pt,line join=round] (175.27, 36.18) --
	(175.27,211.31);

\path[draw=drawColor,line width= 0.6pt,line join=round] ( 37.58, 40.52) --
	(211.31, 40.52);

\path[draw=drawColor,line width= 0.6pt,line join=round] ( 37.58, 89.48) --
	(211.31, 89.48);

\path[draw=drawColor,line width= 0.6pt,line join=round] ( 37.58,138.43) --
	(211.31,138.43);

\path[draw=drawColor,line width= 0.6pt,line join=round] ( 37.58,187.39) --
	(211.31,187.39);

\path[draw=drawColor,line width= 0.6pt,line join=round] ( 67.11, 36.18) --
	( 67.11,211.31);

\path[draw=drawColor,line width= 0.6pt,line join=round] (110.37, 36.18) --
	(110.37,211.31);

\path[draw=drawColor,line width= 0.6pt,line join=round] (153.64, 36.18) --
	(153.64,211.31);

\path[draw=drawColor,line width= 0.6pt,line join=round] (196.90, 36.18) --
	(196.90,211.31);
\definecolor{drawColor}{RGB}{0,191,196}

\path[draw=drawColor,line width= 0.6pt,line join=round] ( 45.48, 44.14) --
	( 53.37, 51.97) --
	( 61.27, 69.50) --
	( 69.17, 97.31) --
	( 77.06,106.90) --
	( 84.96,108.57) --
	( 92.86,109.55) --
	(100.75,108.18) --
	(108.65,109.25) --
	(116.55,109.06) --
	(124.45,108.67) --
	(132.34,105.92) --
	(140.24,110.43) --
	(148.14,109.65) --
	(156.03,110.62) --
	(163.93,109.35) --
	(171.83,116.11) --
	(179.72,110.43) --
	(187.62,116.70) --
	(195.52,107.98) --
	(203.41,109.06);

\path[draw=drawColor,line width= 0.4pt,line join=round,line cap=round] ( 45.48, 44.14) circle (  1.96);

\path[draw=drawColor,line width= 0.4pt,line join=round,line cap=round] ( 53.37, 51.97) circle (  1.96);

\path[draw=drawColor,line width= 0.4pt,line join=round,line cap=round] ( 61.27, 69.50) circle (  1.96);

\path[draw=drawColor,line width= 0.4pt,line join=round,line cap=round] ( 69.17, 97.31) circle (  1.96);

\path[draw=drawColor,line width= 0.4pt,line join=round,line cap=round] ( 77.06,106.90) circle (  1.96);

\path[draw=drawColor,line width= 0.4pt,line join=round,line cap=round] ( 84.96,108.57) circle (  1.96);

\path[draw=drawColor,line width= 0.4pt,line join=round,line cap=round] ( 92.86,109.55) circle (  1.96);

\path[draw=drawColor,line width= 0.4pt,line join=round,line cap=round] (100.75,108.18) circle (  1.96);

\path[draw=drawColor,line width= 0.4pt,line join=round,line cap=round] (108.65,109.25) circle (  1.96);

\path[draw=drawColor,line width= 0.4pt,line join=round,line cap=round] (116.55,109.06) circle (  1.96);

\path[draw=drawColor,line width= 0.4pt,line join=round,line cap=round] (124.45,108.67) circle (  1.96);

\path[draw=drawColor,line width= 0.4pt,line join=round,line cap=round] (132.34,105.92) circle (  1.96);

\path[draw=drawColor,line width= 0.4pt,line join=round,line cap=round] (140.24,110.43) circle (  1.96);

\path[draw=drawColor,line width= 0.4pt,line join=round,line cap=round] (148.14,109.65) circle (  1.96);

\path[draw=drawColor,line width= 0.4pt,line join=round,line cap=round] (156.03,110.62) circle (  1.96);

\path[draw=drawColor,line width= 0.4pt,line join=round,line cap=round] (163.93,109.35) circle (  1.96);

\path[draw=drawColor,line width= 0.4pt,line join=round,line cap=round] (171.83,116.11) circle (  1.96);

\path[draw=drawColor,line width= 0.4pt,line join=round,line cap=round] (179.72,110.43) circle (  1.96);

\path[draw=drawColor,line width= 0.4pt,line join=round,line cap=round] (187.62,116.70) circle (  1.96);

\path[draw=drawColor,line width= 0.4pt,line join=round,line cap=round] (195.52,107.98) circle (  1.96);

\path[draw=drawColor,line width= 0.4pt,line join=round,line cap=round] (203.41,109.06) circle (  1.96);
\definecolor{drawColor}{RGB}{248,118,109}

\path[draw=drawColor,line width= 0.6pt,line join=round] ( 45.48, 48.06) --
	( 53.37, 55.21) --
	( 61.27, 72.83) --
	( 69.17,103.48) --
	( 77.06,171.72) --
	( 84.96,195.22) --
	( 92.86,195.22) --
	(100.75,195.32) --
	(108.65,195.91) --
	(116.55,195.32) --
	(124.45,195.03) --
	(132.34,195.42) --
	(140.24,195.32) --
	(148.14,196.40) --
	(156.03,195.81) --
	(163.93,198.06) --
	(171.83,202.08) --
	(179.72,203.35) --
	(187.62,198.75) --
	(195.52,195.52) --
	(203.41,196.10);

\path[draw=drawColor,line width= 0.4pt,line join=round,line cap=round] ( 45.48, 48.06) circle (  1.96);

\path[draw=drawColor,line width= 0.4pt,line join=round,line cap=round] ( 53.37, 55.21) circle (  1.96);

\path[draw=drawColor,line width= 0.4pt,line join=round,line cap=round] ( 61.27, 72.83) circle (  1.96);

\path[draw=drawColor,line width= 0.4pt,line join=round,line cap=round] ( 69.17,103.48) circle (  1.96);

\path[draw=drawColor,line width= 0.4pt,line join=round,line cap=round] ( 77.06,171.72) circle (  1.96);

\path[draw=drawColor,line width= 0.4pt,line join=round,line cap=round] ( 84.96,195.22) circle (  1.96);

\path[draw=drawColor,line width= 0.4pt,line join=round,line cap=round] ( 92.86,195.22) circle (  1.96);

\path[draw=drawColor,line width= 0.4pt,line join=round,line cap=round] (100.75,195.32) circle (  1.96);

\path[draw=drawColor,line width= 0.4pt,line join=round,line cap=round] (108.65,195.91) circle (  1.96);

\path[draw=drawColor,line width= 0.4pt,line join=round,line cap=round] (116.55,195.32) circle (  1.96);

\path[draw=drawColor,line width= 0.4pt,line join=round,line cap=round] (124.45,195.03) circle (  1.96);

\path[draw=drawColor,line width= 0.4pt,line join=round,line cap=round] (132.34,195.42) circle (  1.96);

\path[draw=drawColor,line width= 0.4pt,line join=round,line cap=round] (140.24,195.32) circle (  1.96);

\path[draw=drawColor,line width= 0.4pt,line join=round,line cap=round] (148.14,196.40) circle (  1.96);

\path[draw=drawColor,line width= 0.4pt,line join=round,line cap=round] (156.03,195.81) circle (  1.96);

\path[draw=drawColor,line width= 0.4pt,line join=round,line cap=round] (163.93,198.06) circle (  1.96);

\path[draw=drawColor,line width= 0.4pt,line join=round,line cap=round] (171.83,202.08) circle (  1.96);

\path[draw=drawColor,line width= 0.4pt,line join=round,line cap=round] (179.72,203.35) circle (  1.96);

\path[draw=drawColor,line width= 0.4pt,line join=round,line cap=round] (187.62,198.75) circle (  1.96);

\path[draw=drawColor,line width= 0.4pt,line join=round,line cap=round] (195.52,195.52) circle (  1.96);

\path[draw=drawColor,line width= 0.4pt,line join=round,line cap=round] (203.41,196.10) circle (  1.96);
\end{scope}
\begin{scope}
\path[clip] (  0.00,  0.00) rectangle (216.81,216.81);
\definecolor{drawColor}{gray}{0.30}

\node[text=drawColor,anchor=base east,inner sep=0pt, outer sep=0pt, scale=  1.20] at ( 32.63, 36.39) {0};

\node[text=drawColor,anchor=base east,inner sep=0pt, outer sep=0pt, scale=  1.20] at ( 32.63, 85.34) {5};

\node[text=drawColor,anchor=base east,inner sep=0pt, outer sep=0pt, scale=  1.20] at ( 32.63,134.30) {10};

\node[text=drawColor,anchor=base east,inner sep=0pt, outer sep=0pt, scale=  1.20] at ( 32.63,183.26) {15};
\end{scope}
\begin{scope}
\path[clip] (  0.00,  0.00) rectangle (216.81,216.81);
\definecolor{drawColor}{gray}{0.20}

\path[draw=drawColor,line width= 0.6pt,line join=round] ( 34.83, 40.52) --
	( 37.58, 40.52);

\path[draw=drawColor,line width= 0.6pt,line join=round] ( 34.83, 89.48) --
	( 37.58, 89.48);

\path[draw=drawColor,line width= 0.6pt,line join=round] ( 34.83,138.43) --
	( 37.58,138.43);

\path[draw=drawColor,line width= 0.6pt,line join=round] ( 34.83,187.39) --
	( 37.58,187.39);
\end{scope}
\begin{scope}
\path[clip] (  0.00,  0.00) rectangle (216.81,216.81);
\definecolor{drawColor}{gray}{0.20}

\path[draw=drawColor,line width= 0.6pt,line join=round] ( 67.11, 33.43) --
	( 67.11, 36.18);

\path[draw=drawColor,line width= 0.6pt,line join=round] (110.37, 33.43) --
	(110.37, 36.18);

\path[draw=drawColor,line width= 0.6pt,line join=round] (153.64, 33.43) --
	(153.64, 36.18);

\path[draw=drawColor,line width= 0.6pt,line join=round] (196.90, 33.43) --
	(196.90, 36.18);
\end{scope}
\begin{scope}
\path[clip] (  0.00,  0.00) rectangle (216.81,216.81);
\definecolor{drawColor}{gray}{0.30}

\node[text=drawColor,anchor=base,inner sep=0pt, outer sep=0pt, scale=  1.20] at ( 67.11, 22.97) {1e+01};

\node[text=drawColor,anchor=base,inner sep=0pt, outer sep=0pt, scale=  1.20] at (110.37, 22.97) {1e+03};

\node[text=drawColor,anchor=base,inner sep=0pt, outer sep=0pt, scale=  1.20] at (153.64, 22.97) {1e+05};

\node[text=drawColor,anchor=base,inner sep=0pt, outer sep=0pt, scale=  1.20] at (196.90, 22.97) {1e+07};
\end{scope}
\begin{scope}
\path[clip] (  0.00,  0.00) rectangle (216.81,216.81);
\definecolor{drawColor}{RGB}{0,0,0}

\node[text=drawColor,anchor=base,inner sep=0pt, outer sep=0pt, scale=  1.40] at (124.45,  8.22) {\bfseries $\gamma$};
\end{scope}
\begin{scope}
\path[clip] (  0.00,  0.00) rectangle (216.81,216.81);
\definecolor{drawColor}{RGB}{0,0,0}

\node[text=drawColor,rotate= 90.00,anchor=base,inner sep=0pt, outer sep=0pt, scale=  1.40] at ( 15.16,123.75) {\bfseries time[s]};
\end{scope}
\begin{scope}
\path[clip] (  0.00,  0.00) rectangle (216.81,216.81);
\definecolor{fillColor}{RGB}{255,255,255}

\path[fill=fillColor] (131.77, 36.18) rectangle (211.31, 81.59);
\end{scope}
\begin{scope}
\path[clip] (  0.00,  0.00) rectangle (216.81,216.81);
\definecolor{fillColor}{gray}{0.95}

\path[fill=fillColor] (137.27, 56.13) rectangle (151.72, 70.59);
\end{scope}
\begin{scope}
\path[clip] (  0.00,  0.00) rectangle (216.81,216.81);
\definecolor{drawColor}{RGB}{248,118,109}

\path[draw=drawColor,line width= 0.6pt,line join=round] (138.71, 63.36) -- (150.28, 63.36);
\end{scope}
\begin{scope}
\path[clip] (  0.00,  0.00) rectangle (216.81,216.81);
\definecolor{drawColor}{RGB}{248,118,109}

\path[draw=drawColor,line width= 0.4pt,line join=round,line cap=round] (144.50, 63.36) circle (  1.96);
\end{scope}
\begin{scope}
\path[clip] (  0.00,  0.00) rectangle (216.81,216.81);
\definecolor{drawColor}{RGB}{248,118,109}

\path[draw=drawColor,line width= 0.6pt,line join=round] (138.71, 63.36) -- (150.28, 63.36);
\end{scope}
\begin{scope}
\path[clip] (  0.00,  0.00) rectangle (216.81,216.81);
\definecolor{drawColor}{RGB}{248,118,109}

\path[draw=drawColor,line width= 0.4pt,line join=round,line cap=round] (144.50, 63.36) circle (  1.96);
\end{scope}
\begin{scope}
\path[clip] (  0.00,  0.00) rectangle (216.81,216.81);
\definecolor{fillColor}{gray}{0.95}

\path[fill=fillColor] (137.27, 41.68) rectangle (151.72, 56.13);
\end{scope}
\begin{scope}
\path[clip] (  0.00,  0.00) rectangle (216.81,216.81);
\definecolor{drawColor}{RGB}{0,191,196}

\path[draw=drawColor,line width= 0.6pt,line join=round] (138.71, 48.91) -- (150.28, 48.91);
\end{scope}
\begin{scope}
\path[clip] (  0.00,  0.00) rectangle (216.81,216.81);
\definecolor{drawColor}{RGB}{0,191,196}

\path[draw=drawColor,line width= 0.4pt,line join=round,line cap=round] (144.50, 48.91) circle (  1.96);
\end{scope}
\begin{scope}
\path[clip] (  0.00,  0.00) rectangle (216.81,216.81);
\definecolor{drawColor}{RGB}{0,191,196}

\path[draw=drawColor,line width= 0.6pt,line join=round] (138.71, 48.91) -- (150.28, 48.91);
\end{scope}
\begin{scope}
\path[clip] (  0.00,  0.00) rectangle (216.81,216.81);
\definecolor{drawColor}{RGB}{0,191,196}

\path[draw=drawColor,line width= 0.4pt,line join=round,line cap=round] (144.50, 48.91) circle (  1.96);
\end{scope}
\begin{scope}
\path[clip] (  0.00,  0.00) rectangle (216.81,216.81);
\definecolor{drawColor}{RGB}{0,0,0}

\node[text=drawColor,anchor=base west,inner sep=0pt, outer sep=0pt, scale=  1.20] at (157.22, 59.23) {ADMM};
\end{scope}
\begin{scope}
\path[clip] (  0.00,  0.00) rectangle (216.81,216.81);
\definecolor{drawColor}{RGB}{0,0,0}

\node[text=drawColor,anchor=base west,inner sep=0pt, outer sep=0pt, scale=  1.20] at (157.22, 44.78) {Proposed};
\end{scope}
\end{tikzpicture}